\documentclass[10pt]{article}
\usepackage{amssymb,amsmath,color}
\usepackage{amsthm}
\usepackage{fullpage}
\usepackage{mathrsfs}
\usepackage{units}
\allowdisplaybreaks
\usepackage{color}

\theoremstyle{definition}
\newtheorem{definition}{Definition}

\let\oldsection\section
\renewcommand\section{\setcounter{equation}{0}\oldsection}
\renewcommand\thesection{\arabic{section}}

\renewcommand{\P}{\mathscr{P}}

\newcommand{\be}{\begin{equation}\label}
\newcommand{\ee}{\end{equation}}
\newcommand{\beaa}{\begin{eqnarray}}
\newcommand{\bea}{\beaa\label}
\newcommand{\eea}{\end{eqnarray}}
\newcommand{\nn}{\nonumber}
\newcommand{\intO}{\int_\Omega}
\newcommand{\Om}{\Omega}
\newcommand{\cd}{\cdot}

\newcommand{\ep}{\varepsilon}

\newcommand{\nep}{n_\eta}
\newcommand{\uep}{u_{\eta}}
\newcommand{\cep}{c_{\eta}}
\newcommand{\Sep}{S_{\eta}}
\newcommand{\ntilde}{\widetilde{n}}
\newcommand{\wstarto}{\stackrel{\star}{\rightharpoonup}}

\newcommand{\Lp}{L^p(\Om)}
\newcommand{\Lq}{L^q(\Om)}
\newcommand{\Lin}{L^\infty(\Om)}

\newcommand{\eps}{\varepsilon}
\newcommand{\al}{\alpha}
\newcommand{\Lt}{L^\theta(\Om)}
\newcommand{\delny}{\partial_\nu}
\newcommand{\dOm}{\partial \Om}
\newcommand{\Lap}{\Delta}
\newcommand{\Ombar}{\overline{\Om}}
\newcommand{\my}{\mu}
\newcommand{\na}{\nabla}
\newcommand{\nbar}{\overline{n}}
\newcommand{\calL}{\mathcal{L}}
\newcommand{\bdry}{\big\rvert_{\dOm}}
\newcommand{\Tmax}{T_{max}}

\newcommand{\norm}[2][ ]{\|#2\|_{#1}}
\newcommand{\normm}[2]{\|#2\|_{#1}}

\newcommand{\set}[1]{\{#1\}}
\newcommand{\bigset}[1]{\left\{#1\right\}}
\newcommand{\R}{\mathbb{R}}
\newcommand{\intnt}{\int_0^t}
\newcommand{\io}{\int_\Om}

\newcommand{\calP}{\mathscr{P}}
\newcommand{\Phii}{\varphi}
\newcommand{\embeddedinto}{\hookrightarrow}

\newcommand{\N}{\mathbb{N}}
\newcommand{\sub}{\subset}
\newcommand{\phat}{\widehat{p}}
\newcommand{\supszerot}{\mathop{\sup}\limits_{s\in(0,t)}}
\newcommand{\suptzeroT}{\mathop{\sup}\limits_{t\in(0,T)}}
\newcommand{\fat}{\qquad \mbox{for all }t\in(0,T)}

\theoremstyle{plain}
\newtheorem{thm}{Theorem}[section]
\newtheorem{lem}[thm]{Lemma}
\newtheorem{proposition}[thm]{Proposition}
\theoremstyle{definition}
\newtheorem{dnt}[thm]{Definition}
\theoremstyle{remark}
\newtheorem{remark}[thm]{Remark}
\allowdisplaybreaks

\title{Global classical small-data solutions for a three-dimensional chemotaxis Navier-Stokes system involving matrix-valued
sensitivities}
\author{Xinru Cao \thanks{Insitute for Mathematical Sciences, Renmin University of China, Zhongguancun Str. 59, 100872 Beijing, China; email: caoxinru@gmail.com}
        \quad
       Johannes Lankeit \thanks{Institut f\"ur Mathematik, Universit\"at Paderborn, Warburger Str.100, 33098 Paderborn, Germany; email: johannes.lankeit@math.uni-paderborn.de}}

\begin{document}
\maketitle
\begin{abstract}
\noindent The coupled chemotaxis fluid system
\[
\left\{
\begin{array}{llc}
n_t=\Delta n-\nabla\cdot(n S(x,n,c)\cdot\nabla c)-u\cdot\nabla n, &(x,t)\in \Omega\times (0,T),\\[6pt]
\displaystyle
c_t=\Delta c-nc-u\cdot\nabla c, &(x,t)\in\Omega\times (0,T),\\[6pt]
\displaystyle
u_t=\Delta u-(u\cdot\nabla )u+\nabla P+n\nabla\Phi,\quad \nabla\cdot u=0, &(x,t)\in\Omega\times (0,T),\\[6pt]
\displaystyle
\nabla c\cdot\nu=(\nabla n-nS(x,n,c)\cdot\nabla c)\cdot\nu=0, \;\; u=0,&(x,t)\in \partial\Omega\times (0,T),\\[6pt]
n(x,0)=n_{0}(x),\quad  c(x,0)=c_{0}(x),\quad u(x,0)=u_0(x)
& x\in\Omega,
\end{array}
\right.
\]
where $S\in (C^2(\overline{\Omega}\times [0,\infty)^2))^{N\times N}$,
is considered in a bounded domain $\Omega\subset \mathbb{R}^N$, $N\in\{2,3\}$, with smooth boundary. We show that it
has global classical solutions if the initial data satisfy certain smallness conditions and give decay properties of these solutions.
{\bf Math Subject Classification (2010):} 35K55, 35B40 35Q35, 92C17, 35B35
\end{abstract}

\section{Introduction}\label{introduction}
Even simple life-forms, like certain species of bacteria, can exhibit a complex collective behaviour.
One particular biological mechanism responsible for some instances of such demeanour is that of chemotaxis, where the bacteria adapt their movement according to the concentration gradient of a particular chemical in their neighbourhood.
If this process takes place in a liquid environment, it is not unreasonable to take into account interactions with the surrounding fluid as well.
Indeed, as description for colonies of {\it bacillus subtilis}, chemotactic bacteria that are known to display organized swimming and bioconvection patterns in a fluid habitat \cite{hillesdon_pedley_kessler,metcalfe_pedley,sokolov_goldstein_feldchtein_aranson,dombrowskietal}, the following model has been suggested in \cite{Tuval}:
\begin{equation}\label{eq:system_without_rot}
\left\{
\begin{array}{llc}
 &n_t=\Delta n - \nabla\cdot(\chi(c)n\nabla c) - u\cdot \nabla n,\\[6pt]
 &c_t=\Delta c - nk(c) -u\cdot\nabla c, \\[6pt]
 &u_t=\Delta u -(u\cdot\nabla)u+\nabla P+n\nabla \Phi,\\[6pt]
 &\nabla \cdot u=0,
\end{array}
\right.
\end{equation}
where a prototypical choice for the functions $\chi$ and $k$ is $\chi(c)=const=\chi$ and $k(c)=c$.
Herein, $n$ denotes the unknown population density of bacteria that move in part randomly and in part as directed by chemotactic effects, and are transported by the surrounding fluid; $c$ denotes the concentration of oxygen, which again diffuses and is transported by the fluid, but at the same time is consumed by the bacteria.
The evolution of the velocity field $u$ of the fluid, finally, is governed by the incompressible Navier-Stokes equations, where the bacteria exert influence by means of bouyant forces due to different densities of water with a high concentration of cells versus low concentration. Using the Boussinesq approximation, this effect is incorporated into the model via the gravitational potential $\na \Phi$, $\Phi\in C^{1+\delta}(\Ombar)$ for some $\delta\in(0,1)$ being a given function. 
The usual boundary conditions posed along with initial conditions to complement \eqref{eq:system_without_rot} are
\[
 \delny n = \delny c=0,\qquad  u =0 \qquad \mbox{on }\partial\Om.
\]

Let us remark that in this model the chemoattractant (oxygen) is consumed and not supplied by the bacteria, which is in contrast to the celebrated Keller-Segel system of chemotaxis \cite{KS} and its variants constituting the center of extensive mathematical investigations since the 1970s, see e.g. the surveys \cite{hillen_painter,horstmann_I,BBTW} and references therein.

Since its introduction and first analytical results (asserting the local existence of weak solutions in \cite{Lorz}), also the chemotaxis-fluid system has inspired
several works addressing mainly the question of existence of classical or weak solutions (the works mentioned below) and long-term behaviour of solutions (\cite{DiFrancesco_Lorz_Markowich,Chae_Kang_Lee2,tan_zhang,wk_arma,jiang_wu_zheng,zhang_li_decay}).

Due to the difficulties associated with the Navier-Stokes equations in three-dimensional domains, many of these works focus on the two-dimensional case (\cite{wk_arma,Wk_CTNS_global_largedata,TaoWk_KS-stokes-porous,zhangzheng,zhang_li_decay}) or more favorable 
variants of the model, for example by resorting to the Stokes equation upon neglection of the nonlinear convective term (\cite{Wk_CTNS_global_largedata,duan_xiang}) or by considering nonlinear instead of linear diffusion of the bacteria (\cite{Liu_Lorz,DiFrancesco_Lorz_Markowich,TaoWk_KS-stokes-porous,TaoWk_ctS3d_nonlin,duan_xiang,chung_kang_kim,vorotnikov}) and consider the three-dimensional case under smallness conditions on the initial data (\cite{Duan_Lorz_Markowich,tan_zhang,ye,Chae_Kang_Lee1}). Also in \cite{kozono}, where existence and uniqueness of mild solutions to a model including \eqref{eq:system_without_rot} as a submodel in addition to Keller-Segel-type chemotaxis, are proven for the full space in arbitrary dimensions
, a smallness assumption (in this case, in the scaling invariant space) is required.

Only recently, the existence of global weak solutions to the system \eqref{eq:system_without_rot} with large initial data has been demonstrated for bounded three-dimensional domains in \cite{wk_ctfluid3dnastoexist}, see also \cite{zhang_li} for even milder diffusion effects, followed by studies of the long-term behaviour of any such ``eventual energy solution'' \cite{wk_ns_oxytaxis}, which, namely, become smooth on some interval $[T,\infty)$ and uniformly converge in the large-time-limit. \\

With this model one further peculiar effect is still unaccounted for that can be observed in colonies of {\it Proteus mirabilis}. Colonies of these bacteria form spiralling streams that always wind counterclockwise \cite{xue_budrene_othmer}. A reason underlying this behaviour is that the swimming of the bacteria, like that of the similar species {\it E. coli}, is biased, when they are close to a surface (cf. \cite{swimmingincircles,ecoli_RHS}).
This can be reflected in chemotaxis equations by allowing for a more general, tensor-valued and spatially inhomogeneous chemotactic sensitivity, so that the model reads
\be{0}
\left\{
\begin{array}{llc}
n_t=\Delta n-\nabla\cdot(n S(x,u,v)\cdot\nabla c)-u\cdot\nabla n, &(x,t)\in \Omega\times (0,T),\\[6pt]
c_t=\Delta c-nc-u\cdot\nabla c, &(x,t)\in\Omega\times (0,T),\\[6pt]
u_t=\Delta u-(u\cdot\nabla )u+\nabla P+n\nabla\Phi,  &(x,t)\in\Omega\times (0,T),\\[6pt]
\na\cdot u=0,&(x,t)\in\Omega\times (0,T),
\end{array}
\right.
\ee
where the sensitivity ${S}(x,n,c)=({s}_{i,j})_{N\times N}$ is a matrix-valued function.
Indeed, when in \cite{xue_othmer} a macroscale model for chemotaxis is derived from a velocity jump process rooted in a cell based model incorporating a minimal description of signal transduction in single cells and accounting for this swimming bias, in the chemotaxis term a contribution perpendicular to the concentration gradient appears (\cite[(5.26)]{xue_othmer}). (For tensor-valued sensitivities arising in chemotaxis equations see also \cite[sec. 4.2.1]{othmer_hillen_II} or \cite[eq.(3.3)]{xue}.)

Mathematically, the introduction of these general sensitivities has the disadvantage that it destroys the natural energy structure coming with \eqref{eq:system_without_rot}.
In point of fact, many results concerning global existence of solutions to \eqref{eq:system_without_rot} rely on the use of an energy inequality featuring an upper estimate of
\[
 \frac{d}{dt} \left[ \io n\log n + \frac12 \io \frac{\chi(c)}{k(c)}|\na c|^2 \right] + \io \frac{|\na n|^2}n + \frac14\io \frac{k(c)}{\chi(c)}|D^2\rho(c)|^2
\]
or very similar quantities, where $\rho$ denotes a primitive of $\frac{\chi}{k}$, see \cite[Formula (3.11)]{Wk_CTNS_global_largedata}, \cite[(2.15)]{wk_arma}, \cite[(3.11)]{Duan_Lorz_Markowich} or \cite[(1.12)]{wk_ns_oxytaxis} or \cite[(3.8)]{Chae_Kang_Lee1}. For the derivation of appropriate estimates, more precisely for certain cancellations of contributions of the first and the second term in the brackets to occur, it seems to be essential that the functions $k$ and $\chi$ satisfy conditions like those given in \cite[(1.8)-(1.10)]{wk_arma}, \cite[(1.7)-(1.9)]{Wk_CTNS_global_largedata}, \cite[(A)(iii)]{Duan_Lorz_Markowich}, \cite[(1.8)]{wk_ns_oxytaxis} or even \cite[(AA), (B)]{Chae_Kang_Lee1}. There is next to no hope of transferring such delicate cancellations to the case of functions $\chi$ that are no longer scalar-valued.

Nevertheless, for some instances of such a system including a rotational sensitivity, the existence of solutions could be shown:
The fluid-free system, obtained from \eqref{0} upon setting $u\equiv 0$, possesses global classical solutions for even more general equations modeling the consumption of oxygen if posed in two-dimensional domains and under a smallness condition on initial data $c_0$. In this case, furthermore, these solutions converge to spatially homogeneous equilibria as $t\to\infty$ (\cite{lswx}). Also in the case of degenerate diffusion the existence of global bounded weak solutions was obtained for the two-dimensional fluid-free case in \cite{xinru_sachiko}. For large initial data and higher spatial dimensions, generalized solutions have been shown to exist in \cite{wk_gensolnstotensorvaluedsensitivities}.

In the presence of a Stokes-governed fluid in two-dimensional domains, global generalized solutions that become smooth eventually and stabilize were constructed in \cite{wk_2dstokesrotational}.
The existence of global weak solutions with bounded $n$-component for the full model including Navier-Stokes equations for the fluid in two-dimensional domains is asserted in \cite{sachiko_positiondep} under the assumption of porous-medium-type diffusion with exponent $m>1$ for the bacteria.

In three dimensions, the existence of a global classical solution to the model with a Stokes-governed fluid
was proven in \cite{xinru_yulan} under the hypothesis that $|S|\leq C(1+n)^{-\al}$, with some $C>0$ and $\al>\frac16$. A similar decay assumption on $S$, here with $\al>0$, made it possible to obtain global existence and boundedness of classical solutions for the same  model with  the second equation replaced by $c_t=\Delta c-n+c - u\cdot \na c$ in two-dimensional domains \cite{wang_xiang}.

In \cite{chamoun_saad_talhouk_anisotropic}, the chemotactic sensitivity and the diffusion coefficient for the bacterial motion, both being $n$- and $x$-dependent, were even assumed to vanish for $n=1$. By a semi-discretization procedure, the existence of weak solutions was established for bounded domains of dimension up to four and in the presence of either Navier-Stokes- or Stokes-fluid.

An alternative assumption prompting the existence of weak solutions in the 3D-Stokes-setting is that of nonlinear diffusion of bacteria, that is, with $\Delta n$ replaced by $\nabla\cdot(n^{m-1}\nabla n)$, with an exponent $m>\frac76$, \cite{wk_ctfluid3dnonlineargeneral}.
Also the long-term behaviour of solutions is examined there: they converge to the semi-trivial steady state.

In the present article we consider \eqref{0} without decay assumptions on $S$ and with Navier-Stokes fluid in three-dimensional domains.
The boundary conditions posed will be
\begin{equation}\label{eq:bdrycond}
 \nabla c\cdot\nu=(\nabla u-nS(x,u,v)\cdot\nabla c)\cdot\nu=0,\;\;u=0,\qquad\qquad(x,t)\in \partial\Omega\times (0,\infty),
\end{equation}
where $\nu$ denotes the outer unit normal.
We concentrate on classical solutions and therefore pose a smallness condition on the initial data. We then obtain global existence of classical solutions and exponential convergence to a constant steady state. Unlike the study of mild solutions to a Keller-Segel-Navier-Stokes system in \cite{kozono}, we are concerned with bounded domains and admit non-scalar sensitivities.

The consideration of convergence rates seems to be new in the context of tensor valued (and space-dependent) sensitivities, although convergence rates for solutions of the chemotaxis-fluid model \eqref{eq:system_without_rot} in the full space have been reported in \cite{Duan_Lorz_Markowich} and \cite{tan_zhang} and in \cite{ye} and also, for Stokes fluid, in \cite{Chae_Kang_Lee2}. The only corresponding result for bounded domains, and thus the only one giving exponential decay, is the recent work \cite{zhang_li_decay}, where two-dimensional bounded domains are considered.
In the derivation of decay estimates in \cite{zhang_li_decay}, it was possible to rely on the already established existence (\cite{Wk_CTNS_global_largedata}) and convergence (\cite{wk_arma}) of solutions. Contrasting this, in the present work we additionally have to ensure global existence of the solutions we are working with and will do so by using a continuation argument that has been used in a similar fluid-free context in \cite{W4}. Moreover, our proof will entail an  improvement of the convergence rate of the fluid component if compared to \cite{zhang_li_decay}.

For these tools and the local existence result to be employable, we will first have to restrict our course of action to the case of $S$ vanishing on the boundary. Only in a later step will we approximate fully general sensitivity functions.
With regards to this step, we will give more detailed proofs, which have not been contained in any previous works concerned with rotational sensitivities.
We will focus on the three-dimensional case. However, since it is possible without further labour, we will perform all calculations and state all results for $N\in\set{2,3}$. The only assumption we place on the domain $\Om\subset\R^N$ is that it be bounded with smooth boundary.
Results concerning bounded domains often include a convexity assumption (see e.g. \cite{Wk_CTNS_global_largedata}), which is used to cope with boundary terms stemming from integration by parts when dealing with an energy functional. By arguments relying on estimates from \cite{ishida_seki_yokota_nonconvex} or \cite{mizoguchi_souplet}, it has become possible to remove this assumption (cf. \cite{jiang_wu_zheng} or also \cite{wang_xiang,xinru_yulan,sachiko_positiondep}). Since our approach does not involve such functionals, these terms will not arise in the first place.

In order to formulate our main result, let us briefly introduce the remaining necessary part of the technical framework:
On the sensitivity function $S$ we will impose the conditions
\begin{equation}\label{eq:condS}
S\in C^2\left(\Ombar\times [0,\infty)\times [0,\infty),\R^{N\times N}\right) \quad\mbox{and}\quad |S(x,n,c)|\le C_S \quad \mbox{ for any } (x,n,c)\in \Ombar\times[0,\infty)\times[0,\infty),
\end{equation}
where $C_S$ is a given positive constant.
The initial data are assumed to satisfy
\begin{equation}\label{inidata}
\left\{
\begin{array}{llc}
&n_0\in C^0(\Ombar),\quad n_0\geq 0 \mbox{ on } \Ombar,\\[6pt]
&c_0\in W^{1,q_0}(\Om),\quad c_0>0 \text{ on } \Ombar,\\[6pt]
&u_0\in D(A^\beta),
\end{array}
\right.
\end{equation}
for some $\beta\in (\frac N4,1)$ and $q_0>N$, where $A$ denotes the ($L^2$-realization of the) Stokes operator under Dirichlet boundary conditions in $\Om$.

Here and in the following, we will denote {the first eigenvalue of $A$} by $\lambda_1'$, and by $\lambda_1$ the first nonzero eigenvalue of $-\Delta$ on $\Om$ under Neumann boundary conditions. (For more details on notation and the precise choice of $q_0$ and $\beta$ we refer to Sections \ref{sec:Prelim} and \ref{sec:constandparam} as well as Theorem \ref{thm:main}.)

For $T\in(0,\infty]$ and initial data with the smoothness indicated in \eqref{inidata}, a classical solution of \eqref{0}, \eqref{eq:bdrycond} on $[0,T)$ is a quadruple of functions $(n,c,u,P)$ satisfying \eqref{0} and \eqref{eq:bdrycond} in a pointwise sense as well as $n(\cdot,0)=n_0$, $c(\cdot,0)=c_0$, $u(\cdot,0)=u_0$ and exhibiting the following regularity properties:
\begin{equation}\label{solspace}
\left\{
\begin{array}{llc}
n\in C^0\left(\Ombar\times[0,T)\right)\cap C^{2,1}\left(\Ombar\times(0,T)\right),\\[6pt]
c\in C^0\left(\Ombar\times[0,T)\right)\cap L^\infty\left((0,T);W^{1,q_0}(\Om)\right)\cap C^{2,1}\left(\Ombar\times(0,T)\right),\\[6pt]
u\in C^0\left(\Ombar\times[0,T)\right)\cap L^\infty\left((0,T);D(A^\beta)\right)\cap C^{2,1}\left(\Ombar\times(0,T)\right),\\[6pt]
P\in C^{1,0}\left({\overline{\Omega}}\times(0,T)\right)
\end{array}
\right.
\end{equation}
It is called global solution if $T=\infty$.
The main result will be the following:
\begin{thm}\label{thm:main}
 Let $N\in\set{2,3}$,
 $p_0\in(\frac N2,\infty)$, $q_0\in(N,\infty)$, and $\beta\in(\frac N4,1)$. Let $m>0$, $C_S>0$, $\Phi\in C^{1+\delta}(\Ombar)$ with some $\delta>0$. Then for any $\alpha_1\in\left(0,\min\{m,\lambda_1\}\right)$ and $\alpha_2\in\left(0,\min\set{\alpha_1,\lambda_1'}\right)$ there are $\eps>0$, $C>0$  such that for any initial data $(n_0,c_0,u_0)$ fulfilling \eqref{inidata} and
 \begin{align}\label{init_small}
  \nbar_0=\frac1{|\Om|}\io n_0=m,\quad \norm[L^{p_0}(\Om)]{n_0-\nbar_0}\le \eps,\quad {
  \norm[\Lin]{c_0}\le\eps
  },\quad \norm[L^N(\Om)]{u_0}\le \eps
 \end{align}
 and any function $S$ satisfying \eqref{eq:condS}, system
 \eqref{0} with boundary condition \eqref{eq:bdrycond} and initial data $(n_0,c_0,u_0)$
  has a global classical solution, which moreover satisfies
  \[
   \norm[L^\infty(\Om)]{n(\cdot,t)-\nbar_0} \leq Ce^{-\al_1 t}, \qquad \norm[W^{1,q_0}(\Om)]{ c(\cdot,t)}\leq Ce^{-\al_1 t},\qquad \norm[L^\infty(\Om)]{u(\cdot,t)}\leq Ce^{-\al_2 t}
  \]
 for any $t>0$.
\end{thm}

Condition \eqref{init_small} in Theorem \ref{thm:main} could be replaced by
 \begin{align}\label{init_small'}
{ \nbar_0=\frac1{|\Om|}\io n_0=m,\quad\norm[L^{p_0}(\Om)]{n_0}\le \eps,\quad
  \norm[L^{N}(\Om)]{\nabla c_0}\le\eps,\quad \norm[L^N(\Om)]{u_0}\le \eps}.
 \end{align}
without affecting the validity of the Theorem, thus exchanging conditions asking for the smallness of oxygen concentration and some kind of uniformity in the distribution of bacteria by conditions that indicate smallness of the bacterial concentration and a somewhat homogeneous dispersion of oxygen.
Let us state this alternative variant:
\begin{thm}\label{thm:alternative}
 Let $N\in\set{2,3}$, $p_0\in(\frac N2,N)$, $q_0\in(N,(\frac1{p_0}-\frac1N)^{-1})$, and $\beta\in(\frac N4,1)$. Let $M>0$, $C_S>0$, $\Phi\in C^{1+\delta}(\Ombar)$ with some $\delta>0$. Then there exist $\eps>0$, $m_0<\eps|\Om|^{-\frac1{p_0}}$ such that for any $m>m_0$, any $\alpha_1\in(0,\min\{m,\lambda_1\})$ and $\alpha_2\in(0,\min\set{\alpha_1,\lambda_1'})$ there is $C>0$ such that for any initial data $(n_0,c_0,u_0)$ fulfilling \eqref{inidata}, \eqref{init_small'} and $\norm[L^\infty(\Om)]{c_0}=M$
 and any function $S$ satisfying \eqref{eq:condS}, system
 \eqref{0} with boundary condition \eqref{eq:bdrycond} and initial data $(n_0,c_0,u_0)$
  has a global classical solution, which moreover satisfies
  \[
   \norm[L^\infty(\Om)]{n(\cdot,t)-\nbar_0} \leq Ce^{-\al_1 t}, \qquad \norm[W^{1,q_0}(\Om)]{ c(\cdot,t)}\leq Ce^{-\al_1 t},\qquad \norm[L^\infty(\Om)]{u(\cdot,t)}\leq Ce^{-\al_2 t}
  \]
 for any $t>0$.
\end{thm}

\begin{remark}
 The condition $m_0<\eps|\Om|^{-\frac1{p_0}}$ ensures the
 existence of initial data to which the theorem is applicable. For
 $m>\eps|\Om|^{-\frac1{p_0}}$ the conditions in \eqref{init_small'} cannot be
 satisfied simultaneously.
\end{remark}

We will not give a separate proof for Theorem \ref{thm:alternative} in detail, since it is very similar to that of Theorem \ref{thm:main}.
In Remark \ref{rem:otherconditions} at the end of Section \ref{sec:Thm} we will indicate the necessary changes in the proof; an appropriately adapted version of Lemma \ref{lem:chooseM} will be given in the Appendix.

In order to derive these theorems, we will begin in Section \ref{sec:Prelim} by recalling or providing a local existence result and some useful estimates.
In Section \ref{sec:constandparam}, we will then ensure the applicability of these estimates and fix constants and parameters that will make it possible to prove Proposition \ref{prop:Szero}, which is Theorem \ref{thm:main} for $S=0$ on the boundary. In Section \ref{sec:generalS} we ensure sufficient boundedness in appropriate spaces to pass to the limit in an approximation procedure for more general sensitivity functions so that the last part of that section, finally, can be devoted to the proof of Theorem \ref{thm:main}.

\section{Preliminaries}
\label{sec:Prelim}
The purpose of this section is to provide the ground for estimates needed in the global existence proof.
Due to the central importance of semigroups in this undertaking,
we begin by recalling $L^p$-$L^q$ estimates for the Neumann heat semigroup as given in \cite[Lemma 1.3]{W4}. 
Here and in the following, 
by $(e^{t\Delta})_{t>0}$ we will denote the Neumann heat semigroup in the domain $\Om$.
\begin{lem}\label{lem:heat}
There exist $k_1,...,k_4>0$ which only
depend on $\Om$ and which have the following properties:\\
{\rm(i)} If $1\le q\le p\le\infty$, then
\bea{1.1}
\|e^{t\Delta}w\|_{\Lp}\le k_1\left(1+t^{-\frac{N}{2}
(\frac{1}{q}-\frac{1}{p})}\right)e^{-\lambda_1 t}\|w\|_{\Lq}
\text{ for all }t>0
\eea
holds for all $w\in\Lq$ with $\intO w=0$.\\
{\rm(ii)} If $1\le q\le p\le\infty$, then
\bea{1.2}
\|\nabla e^{t\Delta}w\|_{\Lp}\le k_2\left(1+t^{-\frac{1}{2}-\frac{N}{2}
(\frac{1}{q}-\frac{1}{p})}\right)e^{-\lambda_1 t}\|w\|_{\Lq}
\text{ for all }t>0
\eea
holds for each $w\in\Lq$.\\
{\rm(iii)} If $2\le q\le p\le\infty$, then
\bea{1.3}
\|\nabla e^{t\Delta}w\|_{\Lp}\le k_3\left(1+t^{-\frac{N}{2}(\frac{1}{q}-\frac{1}{p})}\right)e^{-\lambda_1 t}
\|\nabla w\|_{\Lq}
\text{ for all }t>0
\eea
is true for all $w\in W^{1,p}(\Om)$.\\
{\rm(iv)} Let $1<q\le p<\infty$ or $1<q<\infty$ and $p=\infty$, then
\bea{1.4}
\|e^{t\Delta}\nabla\cdot w\|_{\Lp}\le k_4\left(1+t^{-\frac{1}{2}
-\frac{N}{2}\left(\frac{1}{q}-\frac{1}{p}\right)}\right)e^{-\lambda_1 t}\|w\|_{\Lq}
\text{ for all }t>0
\eea
is valid for any $w\in (L^q(\Om))^N$.
\end{lem}
\begin{proof}
This is \cite[Lemma 1.3]{W4}. The parts of Cases (iii) and (iv) which are missing there, are proven in \cite[Lemma 2.1]{Cao2014Global}.
\end{proof}

Because of the third equation in \eqref{0}, the Neumann Laplacian is not the only operator generating a semigroup which is important for analyzing the solutions of \eqref{0}. Before introducing the Stokes operator and recalling estimates for the corresponding semigroup, however, let us briefly familiarize ourselves with  the appropriate spaces.

For $p\in (1,\infty)$ the spaces of solenoidal vector fields are defined as the $L^p$-closure of the set of divergence-free smooth vector fields:
\[
 L^p_{\sigma}(\Om) = \overline{C_{0,\sigma}^\infty(\Om,\R^N)}^{\norm[L^p(\Om)]{\cdot}}= \overline{\set{\Phii\in C_0^\infty(\Om,\R^N); \na\cdot \Phii=0}}^{\norm[L^p(\Om)]{\cdot}}.
\]
Indeed, the space $L^p(\Om,\R^N)$ is the direct sum of this solenoidal space and a space $\set{\nabla \Phii; \Phii\in W^{1,p}(\Om)}$ consisting of gradients and there exists a projection from $L^p(\Om,\R^N)$ onto $L^p_\sigma(\Om)$, the so-called Helmholtz projection $\mathscr {P}$. More precisely, we have the following:
\begin{lem}\label{lem:helmholtzbounded}
 The Helmholtz projection $\calP $ defines a bounded linear operator $\calP \colon L^p(\Om,\R^N)\to L^p_\sigma(\Om)$;
 in particular, for any $p\in (1,\infty)$ there is $k_5(p)>0$ such that
 \[
  \norm[L^p(\Om)]{\calP  w}\leq k_5(p)\norm[L^p(\Om)]{w}
 \]
 for every $w\in L^p(\Om)$.
\end{lem}
\begin{proof}
 See \cite[Thm. 1 and Thm. 2]{fujiwara_morimoto}.
\end{proof}
The Stokes operator on $L^p_\sigma(\Om)$ is defined as $A_p=-\calP \Lap$ with domain $D(A_p)=W^{2,p}(\Om)\cap W^{1,p}_0(\Om) \cap L^p_\sigma(\Om)$.  Since $A_{p_1}$ and $A_{p_2}$ coincide on the intersection of their domains for $p_1,p_2\in(1,\infty)$, we will drop the index $p$ in the following without fearing confusion.
This operator generates a semigroup for which estimates similar to the previous ones hold:

\begin{lem}\label{lem:stokes}
The Stokes operator $A$ generates the analytic semigroup $(e^{-tA})_{t>0}$ in $L^r_{\sigma}(\Om)$. Its spectrum satisfies $\lambda_1':=\inf {\rm Re}\, \sigma(A)>0$ and we fix $\mu\in(0,\lambda_1')$. For any such $\mu$, the following holds:\\
(i) For any $p\in(1,\infty)$ and $\gamma\geq 0$ there is $k_6(p,\gamma)>0$ such that
\beaa\label{2.0}
\norm[L^p(\Om)]{A^\gamma e^{-tA} \phi} \leq k_6(p,\gamma) t^{-\gamma}e^{-\mu t}\norm[L^p(\Om)]{\phi}
\eea
holds for all $t>0$ and all $\phi\in L^p_\sigma(\Om)$.\\
(ii) For $p,q$ satisfying $1< p\le q<\infty$ there exists $k_7(p,q)>0$ such that
\beaa\label{2.1}
\|e^{-tA}\phi\|_{\Lq}\le k_7(p,q) t^{-\frac{N}{2}\left(\frac{1}{p}-\frac{1}{q}\right)}e^{-\mu t} \|\phi\|_{\Lp}
\eea
holds for all $t>0$ and all $\phi\in L^p_\sigma(\Om)$.\\
(iii) For any $p,q$ with $1< p\le q<\infty$ there is $k_8(p,q)>0$ such that for all $t>0$ and $\phi\in L^p_\sigma(\Om)$
\bea{2.4'}
\|\nabla e^{-tA}\phi\|_{L^q(\Om)}\le k_8(p,q)t^{-\frac{1}{2}-\frac{N}{2}\left(\frac{1}{p}-\frac{1}{q}\right)}e^{-\mu t}\|\phi\|_{\Lp}.
\eea
(iv) If $\gamma\geq0$ and $1<q<p<\infty$ satisfy $2\gamma-\frac Nq\geq 1-\frac Np$, then there is $k_9(\gamma,p,q)$ such that for all $\phi\in D(A^\gamma_q)$
\begin{equation}\label{eq:stokesdomainembeddingw}
 \norm[W^{1,p}(\Om)]{\phi} \leq k_9(\gamma,p,q) \norm[L^q(\Om)]{A^\gamma \phi}.
\end{equation}
\end{lem}
\begin{proof} That $A$ generates an analytic semigroup in $L^r_\sigma(\Om)$ was shown in \cite{giga_stokesanalyticinLr}. The estimate in (i) for its  fractional powers is a consequence of this fact, see \cite[Def. 1.4.7 and Theorem 1.4.3]{henry}. Estimates like those in (ii) and (iii) constitute another well-known property of the Stokes semigroup, see e.g. \cite[Chapter 6]{Wie} or \cite[p.201]{giga_semilinear}. They can be proven by combining the Sobolev type embedding theorem and an embedding result for domains of fractional powers of $A$  with estimates as in (i).
Namely, according to \cite[Prop. 1.4]{giga_miyakawa}, $D(A_r^\gamma)\embeddedinto H_r^{2\gamma}$ for any $\gamma\geq 0$, where $H_r^{2\gamma}=F_{r,2}^{2\gamma}$ is a Bessel potential space. Such spaces are covered by the embedding theorem \cite[Thm. 3.3.1 (ii)]{Triebel}, which states that
$F^{s_0}_{p_0,q_0}(\Om)\embeddedinto F_{p_1,q_1}^{s_1}(\Om)$, if $s_0-\frac{n}{p_0} \geq s_1 - \frac{n}{p_1}$, $0<p_0<\infty$, $0<p_1<\infty$, $0<q_0\leq\infty$, $0<q_1\leq \infty$ and $-\infty<s_1<s_0<\infty$. In particular,
\[
  D(A^{\frac n2(\frac1p-\frac1q)}_p)\embeddedinto H_p^{n(\frac1p-\frac1q)}(\Om) = F_{p,2}^{n(\frac1p-\frac1q)}(\Om)\embeddedinto F_{q,2}^0(\Om)=L^q(\Om)
\]
and analogously $D(A^{\frac12+\frac n2(\frac1p-\frac1q)})\embeddedinto W^{1,q}(\Om)$, so that an application of (i) yields (ii) and (iii), respectively. The same embedding results also readily ensure the validity of (iv).
%
%
%
%
%
%
%
%
\end{proof}

The following lemma, giving elementary estimates for integrals that arise in calculations involving semigroup representations of solutions, will find frequent use in the proof of Proposition \ref{prop:Szero}.
\begin{lem}\label{lem:integralestimates}
For all $\eta>0$ there is $C=C(\eta)>0$ such that for all $\al\in [0,1-\eta]$, $\beta\in[\eta,1-\eta]$, $\gamma,\delta\in \R$ satisfying $\frac1\eta\geq\gamma-\delta\geq\eta$ and for all $t>0$, we have
\[
 \int_0^t \left(1+s^{-\al}\right)\left(1+(t-s)^{-\beta}\right) e^{-\gamma s} e^{- \delta(t-s)} ds \leq C(\eta) e^{-\min\{\gamma,\delta\} t} \left(1+t^{\min\set{0, 1-\alpha-\beta}}\right).
\]
\end{lem}
\begin{proof}
Since the statement is a minimally sharpened version of \cite[Lemma 1.2]{W4}, it is not surprisig that its proof can be performed along the same lines as in \cite[Lemma 1.2]{W4}. We include a proof in the appendix.
\end{proof}
\begin{remark}
The roles of $\delta$ and $\gamma$ can of course be exchanged if those of $\alpha$ and $\beta$ are. 
The constant $C(\eta)$ becomes unbounded as $\eta\to 0^+$.
\end{remark}

In cases where the previous lemma yields another than the desired exponent, the following elementary fact may be of use:

\begin{lem}
 \label{lem:improveintestimate}
 Let $0\ge a\ge b$ and $t>0$. Then $(1+t^a)\leq 2 (1+t^b)$.
\end{lem}
\begin{proof}
 If $t>1$, then $1+t^a\le 2\le 2+2t^b$. If $t\le 1$, by the nonnegativity of $a-b$ the inequality $t^{a-b}\le 1^{a-b}$ holds and hence $1+t^a\le 1+t^b=1+t^b t^{a-b}<2(1+t^b)$.
\end{proof}

Another similarly elementary observation is the following:
\begin{lem}\label{lem:product}
 Let either $a, b\geq 0$ or $a, b\leq 0$. Then for any $t>0$, the inequality $(1+t^a)(1+t^b)\leq 3(1+t^{a+b})$ holds.
\end{lem}
\begin{proof}
 If $a, b\geq 0$, for $t\geq 1$, we have $t^a\leq t^{a+b}\leq 1+t^{a+b}$, whereas for $t\leq 1$, $t^a\leq 1\leq 1+t^{a+b}$. The same estimates hold for $t^b$, and thus $(1+t^a)(1+t^b)= 1 + t^a+t^b+t^{a+b}\leq 3(1+t^{a+b})$. For $a, b<0$, one has to exchange the cases $t\geq 1$ and $t\leq 1$.
\end{proof}

As final preparatory step, we include the following result on local existence of solutions:

\begin{lem}\label{lem:locexistence}
 Let $N\in\set{2,3}$, $q>N$, $\beta\in(\frac N4,1)$ and $C_S>0$ and let $S$ be a function satisfying \eqref{eq:condS}. In addition assume that there exists a compact set $K\subset \Om$ such that
 \begin{equation}\label{eq:Szeroatbdry}
  S(x,n,c) = 0 \quad \mbox{ for any }n\geq 0, c\geq 0, x\in \Om\setminus K.
 \end{equation}
 Assume that $(n_0,c_0,u_0)$ satisfy \eqref{inidata}. \\
 (i) There exist
\begin{align*}
\tau=&\tau(q,\beta,\norm[L^\infty(\Om)]{n_0}, \norm[W^{1,q}(\Om)]{c_0},\norm[L^2(\Om)]{A^\beta u_0},C_S)>0 \qquad \mbox{and }\\
\Gamma=&\Gamma(q,\beta,\norm[L^\infty(\Om)]{n_0}, \norm[W^{1,q}(\Om)]{c_0},\norm[L^2(\Om)]{A^\beta u_0},C_S)>0
\end{align*}
(where for fixed $\beta$ and $q$ the value of ~$\Gamma$ is nondecreasing in the arguments $\norm[L^\infty(\Om)]{n_0}$, $\norm[W^{1,q}(\Om)]{c_0}$, $\norm[L^2(\Om)]{A^\beta u_0}$, $C_S$, and $\tau$ is nonincreasing with respect to them) and a classical solution $(n,c,u,P)$ of \eqref{0}, \eqref{eq:bdrycond} on $[0,\tau]$ with initial data $(n_0,c_0,u_0)$ which satisfies
 \[
  \norm[L^\infty(\Om)]{n(\cdot,t)}+\norm[W^{1,q}(\Om)]{c(\cdot,t)}+\norm[D(A^\beta)]{u(\cdot,t)}\leq  \Gamma
  \mbox{ for every } t\in [0,\tau].
 \]
 (ii) This solution can be extended to a maximal time interval, more precisely: There are $\Tmax>0$ and a classical solution $(n,c,u,P)$ of \eqref{0} in $\Om\times[0,\Tmax)$ such that
\begin{equation}\label{eq:blowupcrit}
 \mbox{if }\Tmax<\infty, \mbox{then } \norm[L^\infty(\Om)]{n(\cdot,t)}+\norm[W^{1,q}(\Om)]{c(\cdot,t)}
 +\norm[L^2(\Om)]{A^{\beta} u(\cdot,t)} \to \infty \mbox{ as } t\nearrow T_{max}.
\end{equation}
Moreover, we have $n>0$ and $c>0$ on $\Om\times(0,\Tmax)$.
For any $T\in(0,T_{max})$, this solution is unique among all functions satisfying \eqref{solspace}, up to addition of functions $\phat$, such that $\phat(\cdot,t)$ is constant for any $t\in(0,T)$ to $P$.
\end{lem}
\begin{proof}
 Condition \eqref{eq:Szeroatbdry} removes any nonlinearity or inhomogeneity from the boundary condition \eqref{eq:bdrycond}. Thus, a
 proof for a very similar system can be found in \cite[Lemma 2.1, p. 324-328]{Wk_CTNS_global_largedata}, where this is shown by means of a Banach fixed-point argument. Differences mainly stem from the presence of $S$, which can be estimated in the Frobenius norm by $C_S$ whenever necessary, so that the reasoning there can almost word by word be applied to the current setting.
\end{proof}

\section{Constants and parameters}\label{sec:constandparam}
Given $m$, $N$, $p_0$, $q_0$, $\beta$, $\al_1$ and $\al_2$ as in Theorem \ref{thm:main}, in this section we shall, mainly by application of Lemma  \ref{lem:integralestimates}, produce constants $C_1,\ldots, C_8$
(which, accordingly, will only depend on $m>0$, $N$, $p_0$, $q_0$, $\beta$ and $\alpha_1$, $\alpha_2$) to be used in the continuation argument in the proof of Proposition \ref{prop:Szero}.
We let $k_1,\ldots, k_9$ denote the constants appearing in the estimates of Lemma \ref{lem:heat}, Lemma \ref{lem:helmholtzbounded} and Lemma \ref{lem:stokes}.
As stated before, $\lambda'_1$ and $\lambda_1$ will be used to refer to the smallest positive eigenvalues of the Stokes operator or the Neumann Laplacian in $\Om$.
As in Proposition \ref{prop:Szero} (or Theorem \ref{thm:main}), we will rely on
\begin{align}
m>0,\\
\label{eq:dim} N\in\set{2,3},\\
\label{eq:p_0} \frac N2<p_0<N,\\
\label{eq:p0q0} q_0>N \mbox{ and } \frac1{q_0}>\frac1{p_0}-\frac1N,\\
\label{eq:choosebeta} \frac N4<\beta<1,\\
\label{eq:alpha1} \alpha_1\in (0,\min\set{m,\lambda_1}),\\
\label{eq:alpha2} \alpha_2\in(0. \min\set{\alpha_1,\lambda_1'})
\end{align}
being satisfied, where we have included upper bounds on $p_0$ and $q_0$ in \eqref{eq:p_0} and \eqref{eq:p0q0} that will be used during Section \ref{sec:Thm}.
We pick $\mu\in(\alpha_2,\lambda_1')$ 
and will henceforth apply Lemma \ref{lem:stokes} with this value of $\mu$ only.

We first note some elementary consequences of these choices that are nevertheless important as they make it possible to use Lemma \ref{lem:integralestimates}.
Because $\alpha_2<\min\set{\alpha_1,\mu}$ and
\(-\frac N2(\frac1{p_0}-\frac1{q_0})\in(-\frac12,0)
\)
due to \eqref{eq:p0q0}, Lemma \ref{lem:integralestimates} ensures the existence of $C_1>0$ such that for all $t>0$
\begin{equation}\label{eq:C1}
 \intnt (1+s^{-\frac{N}2(\frac1{p_0}-\frac1{q_0})}) e^{-\mu(t-s)} e^{-\alpha_1 s} ds\leq C_1e^{-\al_2 t}.
\end{equation}
Since
 \(
  -\frac12\in(-1,0)\), \(-1+\frac{N}{2q_0}\in(-1,0)
 \)
 and
 \(
  1-\frac12-1+\frac{N}{2q_0}=-\frac12+\frac{N}{2q_0}<0,
 \)
Lemma \ref{lem:integralestimates} also provides us with $C_2>0$ such that
\begin{equation}\label{eq:C2}
\intnt (t-s)^{-\frac12} (1+s^{-1+\frac{N}{2q_0}})e^{-\mu(t-s)} e^{-\al_2s} ds\leq C_2(1+t^{-\frac12+\frac N{2q_0}})e^{-\al_2 t} \quad \mbox{for all }t>0.
\end{equation}

Because 
\(-\frac N2(\frac1{p_0}-\frac1{q_0})\in(-\frac12,0)
\) by \eqref{eq:p0q0} and $1-\frac{1}{2}-\frac{N}2(\frac1{p_0}-\frac1{q_0})>0>-\frac12$, Lemma \ref{lem:integralestimates} in combination with Lemma \ref{lem:improveintestimate} yields $C_3>0$ satisfying

\begin{equation}\label{eq:C3}
 \intnt(t-s)^{-\frac12}(1+s^{-\frac{N}2(\frac1{p_0}-\frac1{q_0})}) e^{-\mu(t-s)} e^{-\alpha_1 s} ds \leq C_3 (1+t^{-\frac{1}{2}})e^{-\alpha_2 t} \qquad\mbox{for all } t>0.
\end{equation}

 As
 $-\frac12-\frac{N}{2q_0}\in(-1,0)$ due to the choice of $q_0$, $-1+\frac N{2q_0}\in(-1,0)$ and 
 $1-\frac12-\frac{N}{2q_0}-1+\frac N{2q_0}=-\frac12$,
 Lemmata \ref{lem:integralestimates} and \ref{lem:improveintestimate} make it possible to find $C_4>0$ such that for all $t>0$
\begin{equation}\label{eq:C4}
\intnt e^{-\mu (t-s) } (t-s)^{-\frac12-\frac N{2q_0}}(1+s^{-1+\frac N{2q_0}})e^{-2\al_2 s} ds \leq C_4(1+t^{-\frac12})e^{-\al_2 t}.
\end{equation}

Since 
$-\frac{N}{2p_0}\in(-1,0)$
and
\(1-\frac12-\frac N{2p_0}\ge-\frac12\),
Lemma \ref{lem:integralestimates} ensures
 the existence of $C_5>0$ such that for any $q\ge q_0$ and any $t>0$ 
we have
\begin{equation}\label{eq:C5}
 \int_0^t (1+(t-s)^{-\frac{1}{2}})e^{-\lambda_1(t-s)}
 (1+s^{-\frac{N}{2p_0} })
 e^{-\al_1 s}ds \le C_5 (1+
 {t^{-\frac{1}{2}}})e^{-\al_1 t}.
\end{equation}

Moreover,
\( -\frac{1}{2}-\frac{N}{2q_0}\in(-1,0)\) 
since $q_0>N$, and
\(1-\frac12-\frac N{2q_0}-1+\frac N{2q_0}=-\frac12.\)
Hence it is possible to find $C_6>0$ such that for all $t>0$,
\begin{align}\label{eq:C6}
\int_0^t(1+(t-s)^{-\frac{1}{2}-\frac{N}{2q_0}})e^{-\lambda_1(t-s)} (1+
s^{-1+\frac{N}{2q_0}}
)e^{-\al_1 s}ds\le C_6(1+t^{-\frac12})e^{-\al_1 t}.
\end{align}

Finally, for $\theta\geq q_0$,
\(-\frac{1}{2}-\frac{N}{2}(\frac{1}{q_0}-\frac{1}{\theta})
 \in (-\frac12-\frac N{2q_0},-\frac12)\subset(-1,0);
\)
by \eqref{eq:p0q0} also
\(-\frac{1}{2}-\frac{N}{2}(\frac1{p_0}-\frac1{q_0})
 \in(-1,0),\)
and \(1-\frac12-\frac N2(\frac1{q_0}-\frac1\theta)-\frac12 -\frac N2 (\frac1{p_0}-\frac1{q_0}) = -\frac{N}{2}(\frac{1}{p_0}-\frac{1}{\theta})\).
Thus Lemma \ref{lem:integralestimates} provides $C_7>0$ such that for  any $\theta\ge q_0$
\begin{align}\label{eq:C7}
 \int_0^t(1+(t-s)^{-\frac{1}{2}-\frac{N}{2}(\frac{1}{q_0}-\frac{1}{\theta})})e^{-\lambda_1(t-s)}
(1+s^{-\frac{1}{2}-\frac{N}{2}(\frac1{p_0}-\frac1{q_0})})e^{-\al_1 s}ds
\le C_7(1+t^{-\frac{N}{2}(\frac{1}{p_0}-\frac{1}{\theta})})e^{-\al_1 t}
\end{align}
for all $t>0$.

Let
\begin{equation}\label{eq:defsigma}
\sigma:=\int_0^\infty (1+s^{-\frac N{2p_0}}) e^{-\alpha_1 s} ds
\end{equation}
and observe that, by the condition \eqref{eq:p_0} on $p_0$, this is finite.

\begin{lem}\label{lem:chooseM}
Given $m$, $N$, $p_0$, $q_0$, $\beta$, $\al_1$ and $\al_2$ as in Theorem \ref{thm:main}, it is possible to
 choose $M_1, M_2, M_3, M_4>0$ and $\eps>0$ such that
\begin{align}
 &k_7(N,q_0)+k_5(q_0)k_7(q_0,q_0) (M_1+{k_1}) C_1\norm[\Lin]{\na\Phi}+3k_7({\scriptstyle \frac{N}{1+\frac{N}{q_0}}},q_0) k_5({\scriptstyle\frac{N}{1+\frac{N}{q_0}}}) M_3M_4 C_2\eps \leq \frac{M_3}2, \label{eq:M3}\\
 \label{eq:M4}
&k_8(N,N)+k_8(N,N)k_5(N)|\Om|^{\frac{q_0-N}{Nq_0}}(M_1+
{k_1}) C_3\norm[\Lin]{\nabla\Phi} \nn\\
&\quad\quad\quad\quad\quad\quad\quad\quad\quad\quad\quad+
3 k_8({\scriptstyle \frac{1}{\frac1{q_0}+\frac1N}},N) k_5({\scriptstyle \frac1{\frac1{q_0}+\frac1N}}) C_4 M_3M_4\eps \leq \frac{M_4}2,\\
\label{eq:M2}
& {
k_2+C_5k_2 (m+(M_1+{k_1})\eps) e^{(M_1+
{k_1})\sigma\eps}+3k_2M_2M_3 C_6\eps \le \frac{M_2}2},\\
\label{eq:M1}
&3C_SC_7k_4M_2m{|\Om|^{\frac{1}{q_0}}}+3C_SC_7k_4M_2(M_1+
{k_1})\eps+3(M_1+{k_1})C_7k_4M_3\eps\le \frac{M_1}2
\end{align}
hold.
\end{lem}
\begin{proof}
First let $A>0$ and $M_2>0$ be such that
\begin{equation}
{
k_2+C_5 k_2 m e^{A}<\frac{M_2}4.
}
\end{equation}
Then we fix $M_1,M_3,M_4>0$ such that
\begin{equation}
\left\{
\begin{array}{llc}
&3C_SC_7k_4M_2m{|\Om|^{\frac{1}{q_0}}}<\frac{M_1}{4},\\
&k_7(N,q_0)+k_5(q_0)k_7(q_0,q_0) (M_1+{k_1})C_1\norm[\Lin]{\na\Phi}<\frac{M_3}{4},\\
&k_8(N,N)+k_8(N,N)k_5(N)|\Om|^{\frac{q_0-N}{Nq_0}}(M_1+
{k_1}) C_3\norm[\Lin]{\nabla\Phi} <\frac{M_4}{4}.
\end{array}
\right.
\end{equation}
Finally, letting $\ep>0$ small enough satisfying
\begin{align*} \ep <&\min\bigg\{\frac{A}{(M_1+
{k_1})\sigma},\frac{1}{12k_7({\scriptstyle \frac{N}{1+\frac{N}{q_0}}},q_0) k_5({\scriptstyle\frac{N}{1+\frac{N}{q_0}}})M_4 C_2},\frac{1}{12 k_8({\scriptstyle \frac{1}{\frac1{q_0}+\frac1N}},N) k_5({\scriptstyle \frac1{\frac1{q_0}+\frac1N}})M_3 C_4},\\
&\frac{M_2}{
{
4C_5 k_2(M_1+
{k_1}) e^A+12k_2M_2M_3C_6
}}
,\frac{M_1}{12C_7k_4 (M_1+
{k_1}) (C_SM_2+M_3)}\bigg\},
\end{align*}
we can easily check that
\eqref{eq:M1}, \eqref{eq:M2}, \eqref{eq:M3} and \eqref{eq:M4} are true.
\end{proof}

\section{Proof of a special case: Sensitivities vanishing near the boundary}\label{sec:Thm}
This section contains the core of the proof of Theorem \ref{thm:main}, concerning global existence and the convergence estimates both. Nevertheless, for the moment we will restrict ourselves to the situation that the sensitivity function $S$ vanishes close to the boundary.
That has the considerable advantage that the nonlinear boundary conditions posed in \eqref{eq:bdrycond} reduce to classical homogeneous Neumann boundary conditions and the existence theorem (Lemma \ref{lem:locexistence}) and standard results
concerning the heat semigroup (cf. Section \ref{sec:Prelim}) become applicable. The case of more general $S$ will be dealt with in Section \ref{sec:generalS}.

Let us first state what we are going to prove. The main difference between this proposition and Theorem \ref{thm:main} lies in the additional condition on $S$.

\begin{proposition}\label{prop:Szero}
Let $N\in\set{2,3}$, {$p_0\in(\frac N2,N)$}, $q_0\in(N,(\frac1{p_0}-\frac1N)^{-1})$, $q_1\ge q_0$, and $\beta\in(\frac N4,1)$. Let $C_S>0$, $\Phi\in C^{1+\delta}(\Ombar)$ with some $\delta>0$, $m>0$. Then for any $\alpha_1\in(0,\min\{m,\lambda_1\})$ and $\alpha_2\in(0,\min\set{\alpha_1,\lambda_1'})$ there are $\eps>0$, $M_1, M_2, M_3, M_4>0$ as in Lemma \ref{lem:chooseM} and $C_8, C_9, C_{10}, C_{11}>0$ such that the following holds:
  For any initial data $(n_0,c_0,u_0)$ fulfilling \eqref{inidata} as well as $c_0\in W^{1,q_1}(\Om)$ and
  \begin{equation} \label{eq:propinitcond}
   \nbar_0=\frac1{|\Om|}\io n_0=m,\quad \norm[L^{p_0}(\Om)]{n_0-\nbar_0}\le \eps,\quad { \norm[\Lin]{c_0}\le\eps
   },\quad \norm[L^N(\Om)]{u_0}\le \eps,
  \end{equation}
  and any function $S$ satisfying \eqref{eq:condS} and
 \[
  S(x,n,c) = 0 \quad \mbox{ for any }n\geq 0, c\geq 0, x\in \Om\setminus K
 \]
 for some compact set $K\subset \Om$, system  \eqref{0} with boundary condition \eqref{eq:bdrycond} and initial data $(n_0,c_0,u_0)$
   has a global classical solution, which, for any $t>0$, moreover satisfies
\begin{align}\label{eq:estimateslikeindefT}
\|n(\cdot,t)- e^{t\Delta} n_0\|_{L^\theta(\Om)}<&M_1 \eps\left(1+t^{-\frac{N}{2}\left(\frac1{p_0}-\frac1\theta\right)}\right)e^{-\al_1 t} \;\;\forall\;\theta\in[q_0,\infty],\nn\\
\|\nabla c(\cdot,t)\|_{\Lin}\le& M_2 \eps\left(1+t^{-\frac{1}{2}}\right)e^{-\al_1 t},\nn\\
\norm[L^{q_0}(\Om)]{u(\cdot,t)}\leq& M_3 \eps \left(1+t^{-\frac12+\frac{N}{2q_0}}\right)e^{-\al_2 t},\nn\\
\norm[L^{N}(\Om)]{\nabla u(\cdot,t)}\leq& M_4 \eps \left(1+t^{-\frac12}\right)e^{-\al_2 t},
\end{align}
   and
   \begin{align*}
     \norm[L^2(\Om)]{A^\beta u(\cdot,t)}\leq
     { C_8 }e^{-\al_2 t}, \qquad \norm[\Lin]{u(\cdot,t)}\leq
     { C_{9} } e^{-\al_2 t},\\
     \norm[L^\infty(\Om)]{n(\cdot,t)-\nbar_0} \leq  { C_{10} } e^{-\al_1 t}, \qquad \norm[W^{1,q_1}(\Om)]{c(\cdot,t)}\leq
     { C_{11} }e^{-\al_1 t}.
   \end{align*}
\end{proposition}

Lemma \ref{lem:locexistence} asserts that there is a solution to \eqref{0}, which is defined on some interval $[0,\Tmax)$. We will denote this solution by $(n,c,u,P)$ in the following. Our main goal is to prove that $\Tmax=\infty$.
In order to show this and to achieve estimates \eqref{eq:estimateslikeindefT}, we define a number $T>0$ as follows:

\begin{dnt}\label{def:T}
 With $\eps>0$, $M_1, M_2, M_3, M_4>0$, $p_0$, $q_0$, $\alpha_1$ and $\alpha_2$ as in Proposition \ref{prop:Szero}, we let
\begin{equation}\label{eq:T}
T:=\sup\left\{\tilde{T}\in (0,\Tmax)~\left|~
\begin{array}{l}
\|n(\cdot,t)- e^{t\Delta} n_0\|_{L^\theta(\Om)}<M_1 \eps\big(1+t^{-\frac{N}{2}(\frac1{p_0}-\frac1\theta)}\big)e^{-\al_1 t} \;\;\forall\;\theta\in[q_0,\infty],\\[6pt]
\|\nabla c(\cdot,t)\|_{\Lin}\le M_2 \eps\big(1+t^{-\frac{1}{2}}\big)e^{-\al_1 t},\\[6pt]
\norm[L^{q_0}(\Om)]{u(\cdot,t)}\leq M_3 \eps \big(1+t^{-\frac12+\frac{N}{2q_0}}\big)e^{-\al_2 t},\\[6pt]
\norm[L^{N}(\Om)]{\nabla u(\cdot,t)}\leq M_4 \eps \big(1+t^{-\frac12}\big)e^{-\al_2 t}\\[6pt]
\qquad \qquad \qquad \qquad \qquad \qquad \qquad \qquad \qquad \qquad \text{ for all } t\in[0,\tilde{T})\\
\end{array}\right\}.\right.
\end{equation}
\end{dnt}
By Lemma \ref{lem:locexistence}, $T$ is well-defined and positive.
Thus what we want to show is $T=\infty$.
In doing so, we will proceed in several steps and at first derive estimates for the component $n$ that are satisfied on $(0,T)$. We will then
show that all of the estimates mentioned in \eqref{eq:T} hold true with even smaller coefficients on the right hand side than appearing in \eqref{eq:T} and finally conclude that $T=\infty$.
The derivation of these estimates will mainly rely on Lemma \ref{lem:heat}, Lemma \ref{lem:helmholtzbounded} and Lemma \ref{lem:stokes}  by means of the estimates from Section \ref{sec:constandparam} and on the fact that the classical solutions on $(0,T)$ can be represented as
\begin{align}
& n(\cdot,t)=e^{t\Delta }n_0 -\intnt e^{(t-s)\Delta}\big(\nabla\cdot\left(n S(\cdot,n,c) \nabla c\right) + u\cdot\nabla n \big)(\cdot,s)ds,\label{eq:varofconst_n}\\
& c(\cdot,t)=e^{t\Delta }c_0-\intnt e^{(t-s)\Delta}\big( nc+u\cdot\nabla c \big)(\cdot,s)ds,\label{eq:varofconst_c}\\
& u(\cdot,t)=e^{-tA} u_0 - \intnt e^{-(t-s)A}{\calP }\big((u\cdot\nabla )u - n\nabla\Phi\big)(\cdot,s)ds,\label{eq:varofconst_u}
\end{align}
for all $t\in(0,T_{\max})$
as per the variation-of-constants formula.

\begin{lem}\label{lem:prelimestimaten}
Under the assumptions of Proposition \ref{prop:Szero}, for all $\theta\in[q_0,\infty]$ we have
\begin{equation}\label{eq:n_Ltheta}
 \|n(\cdot,t)-\bar{n}_0\|_{L^\theta(\Om)}\le (M_1+
 {k_1}) \eps\left(1+t^{-\frac{N}{2}(\frac{1}{p_0}-\frac{1}{\theta})}\right)e^{- \al_1 t}\;\;\text{ for all }t\in(0,T).
\end{equation}
\end{lem}
\begin{proof}
Since $\nbar_0$ is a constant, $e^{t\Lap}\nbar_0=\nbar_0$ for all $t\in(0,T)$, and moreover due to $\int_{\Om} (n_0-\nbar_0)=0$, Lemma \ref{lem:heat}(i), \eqref{eq:T} and \eqref{eq:propinitcond} show that
\begin{align*}\nn
\|n{(\cdot,t)}-\bar{n}_0\|_{L^\theta(\Om)}&\le\|n{(\cdot,t)}-e^{t\Delta}n_0\|_{\Lt}+\|e^{t\Delta} (n_0-\bar{n}_0)\|_{\Lt}\\\nn
&\le M_1\eps\left(1+t^{-\frac{N}{2}(\frac{1}{p_0}-\frac{1}{\theta})}\right)e^{-\al_1 t}+k_1\left(1+t^{-\frac{N}{2}(\frac{1}{p_0}-\frac{1}{\theta})}\right)e^{-\lambda_1 t}\|n_0-\bar{n}_0\|_{L^{p_0}(\Om)}\\\nn
&\le M_1\eps\left(1+t^{-\frac{N}{2}(\frac{1}{p_0}-\frac{1}{\theta})}\right)e^{-\al_1 t}+
k_1\left(1+t^{-\frac{N}{2}(\frac{1}{p_0}-\frac{1}{\theta})}\right)e^{-\lambda_1 t}\eps
\\
&\le (M_1+
k_1) \eps\left(1+t^{-\frac{N}{2}(\frac{1}{p_0}-\frac{1}{\theta})}\right)e^{- \al_1 t}
\end{align*}
for all $t\in(0,T)$, $\theta\in[q_0,\infty]$.
\end{proof}

\begin{lem}\label{lem:estimatec}
Under the assumptions of Proposition \ref{prop:Szero}, the second component of the solution satisfies
\begin{align}
\label{eq:clinfty}
\|c(\cdot,t)\|_{L^\infty(\Om)}\le e^{(M_1+{k_1})\sigma\eps}\eps e^{-\al_1 t}\qquad\text{ for all }t\in(0,T)
\end{align}
with $\sigma$ taken from \eqref{eq:defsigma}.
\end{lem}
\begin{proof}
We let $p\ge 1$,
multiply the second equation of \eqref{0} by $pc^{p-1}$ and integrate over $\Om$, so that we have
\beaa
\frac{d}{dt}\intO c^p\le -p\intO nc^p \qquad \mbox{ on }(0,T).
\eea
By an obvious pointwise estimate and \eqref{eq:n_Ltheta} with $\theta=\infty$,
\begin{equation}
-n(x,t)\le\|n(\cdot,t)-\bar{n}_0\|_{\Lin}-\bar{n}_0\le \left(M_1+
{k_1}\right)\eps\left(1+t^{-\frac{N}{2p_0}}\right)e^{-\al_1 t}-\bar{n}_0\qquad\mbox{for all } x\in\Om, t\in(0,T).
\end{equation}
Due to the nonnegativity of $pc^p$, we infer that
\beaa
\frac{d}{dt}\intO c^p\le \left((M_1+{k_1})\eps(1+t^{-\frac{N}{2p_0}})e^{-\al_1 t}-\bar{n}_0\right)p\intO c^p
\eea
for all $t\in(0,T)$.
Thus we get
\beaa
\intO c^p\le \exp\bigg({p\int_0^t\left((M_1+{k_1})\eps(1+s^{-\frac{N}{2p_0}})e^{-\al_1 s}-\bar{n}_0\right)ds}\bigg)\intO c_0^p \qquad \mbox{for all }t\in(0,T). 
\eea
Taking the $p$-th root on both sides, we are left with
\begin{align*}
\|c{(\cdot,t)}\|_{\Lp}&\le \|c_0\|_{\Lp} e^{-\bar{n}_0t} \exp\bigg({\eps(M_1+{k_1})\int_0^t (1+s^{-\frac{N}{2p_0}})e^{-\al_1 s}ds}\bigg)\\\nn
&\le  \|c_0\|_{\Lp} e^{-\bar{n}_0t}e^{(M_1+
{k_1})\sigma\eps}\qquad \mbox{for all } t\in(0,T),
\end{align*}
which holds for arbitrary $p\ge 1$ and where $\sigma$ is as defined in \eqref{eq:defsigma}.
In the limit $p\to\infty$, we therefore obtain
\begin{equation}\label{eq:thisiswhereweneedcontrolonc0}
\|c(\cdot,t)\|_{\Lin}\le \norm[\Lin]{c_0}e^{\sigma(M_1+{k_1})\eps}e^{-\bar{n}_0t}
\end{equation}
for all $t\in(0,T)$ and may, due to \eqref{eq:alpha1}, \eqref{eq:propinitcond}, conclude \eqref{eq:clinfty}.
\end{proof}

\begin{lem}\label{lem:estimateu}
Under the assumptions of Proposition \ref{prop:Szero}, the component $u$ of the solution satisfies
\beaa
\norm[L^{q_0}(\Omega)]{u(\cdot,t)}\leq \frac{M_3}2 \eps \left(1+t^{-\frac12+\frac{N}{2q_0}}\right)e^{-\al_2 t}\qquad \text{ for all } t\in(0,T).
\eea
\end{lem}
\begin{proof}
If we use that $\P\nabla\Phi=0$ and apply the triangle inequality in the variation-of-constants formula \eqref{eq:varofconst_u}
for $u$, Lemma \ref{lem:helmholtzbounded}  and Lemma \ref{lem:stokes} (ii) yield
\begin{align*}
 \norm[L^{q_0}(\Om)]{u(\cdot,t)}\leq& k_7(N,q_0) t^{-\frac N2(\frac1N-\frac1{q_0})} e^{-\my t} \norm[L^N(\Om)]{u_0}\\
 &+ \intnt k_7(q_0,q_0)k_5(q_0) e^{-\my (t-s)}\norm[L^{q_0}(\Om)]{n(\cdot,s)-\nbar_0}\norm[\Lin]{\nabla\Phi} ds \\
 &+ \intnt k_7({\scriptstyle \frac{N}{1+\frac{N}{q_0}}},q_0) (t-s)^{-\frac N2\big(\frac{1+{\frac N{q_0}}}N-\frac1{q_0}\big)}e^{-\mu(t-s)}\norm[L^{\frac N{1+{\frac N{q_0}}}}(\Om)]{\calP {(u\cdot\nabla u)(\cdot,s)}} ds\\
 =:& k_7(N,q_0) t^{-\frac12+\frac{N}{2q_0}}e^{-\mu t}\norm[L^N(\Om)]{u_0} +I_1 + I_2
\end{align*}
for all $t\in(0,T)$.
Here an application of estimate \eqref{eq:n_Ltheta} for $\theta=q_0$ and \eqref{eq:C1} in the first integral shows that
\begin{align*}
 I_1 &\leq k_5(q_0)k_7(q_0,q_0) (M_1+k_1)\norm[L^\infty(\Om)]{\na\Phi} \intnt  \eps \left(1+s^{-\frac N2(\frac1{p_0}-\frac1{q_0})}\right)e^{-\mu(t-s)} e^{-\alpha_1 s} ds \\\label{c1}
 &\leq k_5(q_0)k_7(q_0,q_0) (M_1+{k_1})\norm[L^\infty(\Om)]{\na\Phi}\eps C_1e^{-\al_2 t}\\
 &\le k_5(q_0)k_7(q_0,q_0) (M_1+{k_1})\norm[L^\infty(\Om)]{\na\Phi} C_1(1+t^{-\frac{1}{2}+\frac{N}{2q_0}})e^{-\al_2 t}
 \eps
\end{align*}
for all $t\in(0,T)$. An application of H\"older's inequality and Lemma \ref{lem:helmholtzbounded} show that
\begin{align*}\norm[L^{\frac{N}{1+{\frac N{q_0}}}}(\Om)]{\calP {(u\cdot\nabla u)(\cdot,t)}}\leq k_5({\scriptstyle\frac{N}{1+\frac{N}{q_0}}}) \norm[L^{q_0}(\Om)]{u{(\cdot,t)}}\norm[L^N(\Om)]{\nabla u{(\cdot,t)}}\qquad \mbox{for all }t\in(0,T)
\end{align*}
and the estimates for the latter two terms, which are valid by \eqref{eq:T}, give
\begin{align*}
I_2&\leq k_7({\scriptstyle \frac{N}{1+\frac{N}{q_0}}},q_0) k_5({\scriptstyle\frac{N}{1+\frac{N}{q_0}}})
 \intnt (t-s)^{-\frac12} M_3M_4 \eps^2 e^{-\mu(t-s)} (1+s^{-\frac12+\frac{N}{2q_0}})e^{-\al_2s}(1+s^{-\frac12})e^{-\al_2s} ds\\
&\leq  k_7({\scriptstyle \frac{N}{1+\frac{N}{q_0}}},q_0) k_5({\scriptstyle\frac{N}{1+\frac{N}{q_0}}}) M_3M_4\eps^2 \intnt (t-s)^{-\frac12} e^{-\mu(t-s)} 3 (1+s^{-1+\frac{N}{2q_0}}) e^{-2\al_2s}ds\\
&\leq 3k_7({\scriptstyle \frac{N}{1+\frac{N}{q_0}}},q_0) k_5({\scriptstyle\frac{N}{1+\frac{N}{q_0}}}) M_3M_4\eps^2 C_2 \left(1+t^{-\frac{1}{2}+\frac{N}{2q_0}}\right)e^{-\al_2 t}\fat,
\end{align*}
where we have also used Lemma \ref{lem:product} and \eqref{eq:C2}. Hence,
 \begin{align*}
  \norm[L^{q_0}(\Om)]{u{(\cdot,t)}}\leq&
  k_7(N,q_0) t^{-\frac12+\frac{N}{2q_0}}e^{-\mu t}  \eps +
  k_5(q_0)k_7(q_0,q_0) (M_1+{k_1})\norm[\Lin]{\na\Phi} C_1\left(1+t^{-\frac{1}{2}+\frac{N}{2q_0}}\right)e^{-\alpha t} \eps\\
  &+3k_7({\scriptstyle \frac{N}{1+\frac{N}{q_0}}},q_0) k_5({\scriptstyle\frac{N}{1+\frac{N}{q_0}}}) M_3M_4\eps^2 C_2 \left(1+t^{-\frac{1}{2}+\frac{N}{2q_0}}\right)e^{-\al_2 t}\\
 \leq&\bigg(k_7(N,q_0)+k_5(q_0)k_7(q_0,q_0) (M_1+{k_1})\norm[\Lin]{\na\Phi} C_1\\\nn
 &~~~~~~~~~~\qquad+3k_7({\scriptstyle \frac{N}{1+\frac{N}{q_0}}},q_0) k_5({\scriptstyle\frac{N}{1+\frac{N}{q_0}}}) M_3M_4 C_2\eps\bigg)\eps \left(1+t^{-\frac12+\frac{N}{2q_0}}\right)e^{-\al_2 t}\\
 \leq& \frac{M_3}2 \eps \left(1+t^{-\frac12+\frac{N}{2q_0}}\right)e^{-\al_2 t}
 \end{align*}
 for all $t\in(0,T)$, according to \eqref{eq:M3}.
\end{proof}

Also the estimate for the gradient is preserved:
\begin{lem}\label{lem:estimatenau}
Under the assumptions of Proposition \ref{prop:Szero}, we also have
\[
 \norm[L^N(\Om)]{\nabla u(\cdot,t)}\leq \frac\eps2 M_4 \left(1+t^{-\frac12}\right)e^{-\al_2 t},\;\;\text{ for all } t\in(0,T).
\]
\end{lem}
\begin{proof}
Starting from
\[ \na u(\cdot,t)=\na e^{-tA} u_0 +\intnt \na e^{-(t-s)A} \P \left((n(\cdot,s)-\nbar_0)\na \Phi\right)ds +\intnt \na e^{-(t-s)A} \P(u\cdot\na)u(\cdot,s) ds, \quad t\in(0,T),\]
we obtain from Lemma \ref{lem:stokes}(iii), H\"older's inequality, Lemma \ref{lem:helmholtzbounded} and \eqref{eq:n_Ltheta} that
\begin{align*}
\norm[L^N(\Om)]{\nabla u(\cdot,t)}&\le \norm[L^N(\Om)]{\nabla e^{tA}u_0}
+\intnt k_8(N,N)(t-s)^{-\frac12}e^{-\mu(t-s)}k_5(N)\norm[L^N(\Om)]
{(n(\cdot,s)-\nbar_0)\nabla\Phi}ds\\\nn
&~~~~+\intnt k_8({\scriptstyle \frac{1}{\frac1{q_0}+\frac1N}},N)
(t-s)^{-\frac12-\frac N2(\frac1{q_0}+\frac1N-\frac1N)}
e^{-\mu(t-s)} k_5({\scriptstyle \frac1{\frac1{q_0}+\frac1N}})
\norm[{L^{\scriptstyle \frac{1}{\frac1{q_0}+\frac1N}}(\Om)}]
{{(u\cdot\nabla) u}(\cdot,s)}ds
\\\nn
&\leq
k_8(N,N) t^{-\frac12}e^{-\mu t} \norm[L^N(\Om)]{u_0}\\
&~~~~+\intnt k_8(N,N) (t-s)^{-\frac12}|\Om|^{\frac{q_0-N}{Nq_0}}\norm[L^{q_0}(\Om)]{n(\cdot,s)-\nbar_0}\norm[\Lin]{\na \Phi}e^{-\mu  (t-s)}ds\\\nn
&~~~~~~+k_8({\scriptstyle \frac{1}{\frac1{q_0}+\frac1N}},N)\intnt (t-s)^{-\frac12-\frac{N}{2q_0}}
e^{-\mu (t-s) }k_5({\scriptstyle \frac1{\frac1{q_0}+\frac1N}})\norm[L^{q_0}(\Om)]{u(\cdot,s)}
\norm[L^N(\Om)]{\nabla u(\cdot,s)}ds\\\nn
 &\leq k_8(N,N) t^{-\frac12}e^{-\al_2 t} \norm[L^N(\Om)]{u_0} +I_3+I_4 \fat.
\end{align*}
Here by \eqref{eq:C3}, we have
\begin{align*}\nn
I_3&\leq k_8(N,N)k_5(N)|\Om|^{\frac{q_0-N}{Nq_0}}(M_1+{k_1})\norm[\Lin]{\nabla\Phi} \eps \intnt (t-s)^{-\frac12}(1+s^{-\frac N2(\frac1{p_0}-\frac1{q_0})})e^{-\mu  (t-s)}e^{-\alpha_1 s} ds\\
&\leq k_8(N,N)k_5(N)|\Om|^{\frac{q_0-N}{Nq_0}}(M_1+{k_1}) \norm[\Lin]{\nabla\Phi} \eps C_3 (1+t^{-\frac12})e^{-\al_2 t}\qquad \mbox{for all }t\in(0,T).
\end{align*}

Furthermore, by Lemma \ref{lem:product} and \eqref{eq:C4},
\begin{align*}
 I_4&\leq \eps^2 M_3M_4  k_8({\scriptstyle \frac{1}{\frac1{q_0}+\frac1N}},N) k_5({\scriptstyle \frac1{\frac1{q_0}+\frac1N}})
 \intnt e^{-\mu (t-s)}(t-s)^{-\frac12-\frac{N}{2q_0}}(1+s^{-\frac12+\frac N{2q_0}}) (1+s^{-\frac12}) e^{-2\al_2s} ds\\
&\le 3 \eps^2 M_3M_4  k_8({\scriptstyle \frac{1}{\frac1{q_0}+\frac1N}},N) k_5({\scriptstyle \frac1{\frac1{q_0}+\frac1N}})
 \intnt e^{-\mu (t-s)}(t-s)^{-\frac12-\frac{N}{2q_0}}(1+s^{-1+\frac N{2q_0}})e^{-2\al_2s} ds\\
&\le 3 \eps^2 M_3M_4  k_8({\scriptstyle \frac{1}{\frac1{q_0}+\frac1N}},N) k_5({\scriptstyle \frac1{\frac1{q_0}+\frac1N}}) C_4 \left(1+t^{-\frac12}\right)e^{-\al_2t} \qquad \mbox{for all } t\in(0,T).
 \end{align*}
And thus finally, thanks to the above estimate and (\ref{eq:M4}), we arrive at
 \begin{align*}
  \norm[N]{\nabla {u(\cdot,t)}}\leq&
  k_8(N,N) t^{-\frac12}e^{-\mu t} \eps
  +k_8(N,N)k_5(N)|\Om|^{\frac{q_0-N}{Nq_0}}(M_1+{k_1}) \norm[\infty]{\nabla\Phi} \eps C_3 \left(1+t^{-\frac12}\right)e^{-\al_2 t}\\\nn
 &~~~~
  +3 \eps^2 M_3M_4  k_8({\scriptstyle \frac{1}{\frac1{q_0}+\frac1N}},N) k_5({\scriptstyle \frac1{\frac1{q_0}+\frac1N}}) C_4 \left(1+t^{-\frac12}\right)e^{-\al_2 t}\\
  &\leq\bigg(k_8(N,N)+k_8(N,N)k_5(N)|\Om|^{\frac{q_0-N}{Nq_0}}(M_1+{k_1}) C_3\norm[\Lin]{\nabla\Phi}\\
  &\qquad\qquad\quad\qquad~
+3 k_8({\scriptstyle \frac{1}{\frac1{q_0}+\frac1N}},N) k_5({\scriptstyle \frac1{\frac1{q_0}+\frac1N}}) C_4 M_3M_4\eps
\bigg)\eps \left(1+t^{-\frac12}\right)e^{-\al_2 t}\\
  \leq& \frac{\eps M_4}2\left(1+t^{-\frac12}\right)e^{-\al_2 t}
 \end{align*}
for all $t\in(0,T)$.
\end{proof}

\begin{lem}\label{lem:estimatenac}
Under the assumptions of Proposition \ref{prop:Szero},  we have
\begin{align}\nn
\|\nabla c(\cdot,t)\|_{\Lin}\le \frac{\eps M_2}2\left(1+t^{-\frac{1}{2}}\right)e^{-\al_1 t}
\end{align}
for all $t\in(0,T)$.
\end{lem}
\begin{proof}
If we use the variation-of-constants formula \eqref{eq:varofconst_c} for $c$, we obtain from Lemma \ref{lem:heat}(ii) that
\begin{align}\nn
\|\nabla c(\cdot,t)\|_{\Lin}&\le\|\nabla e^{t\Delta}c_0\|_{\Lin}+\int_0^t\|\nabla e^{(t-s)\Delta}n(\cdot,s)c(\cdot,s)\|_{\Lin}ds\\\nn
&~~~~~~~~~~~~~~~~~~~~~~~~~~~{+}\int_0^t \|\nabla e^{(t-s)\Delta}u(\cdot,s)\cdot\nabla c(\cdot,s)\|_{\Lin}ds\\
&\le { k_2\left(1+t^{-\frac 12}\right)e^{-\lambda_1 t}\| c_0\|_{\Lin}}+I_5+I_6 \qquad \mbox{ on }(0,T). \label{eq:thisischangedbyothercond}
\end{align}
In the first integral we can again apply Lemma \ref{lem:heat}(ii), which gives
\begin{align}\nn
I_5&\le\int_0^t k_2(1+(t-s)^{-\frac{1}{2}})e^{-\lambda_1(t-s)}\|n(\cdot,s)c(\cdot,s)\|_{\Lin}ds\\\nn
 &{\le} \int_0^t k_2(1+(t-s)^{-\frac{1}{2}})e^{-\lambda_1(t-s)}{\|n(\cdot,s)\|_{\Lin}}\|c(\cdot,s)\|_{\Lin}ds
\end{align}
on $(0,T)$.  At this point, Lemma \ref{lem:prelimestimaten}, Lemma \ref{lem:estimatec} and \eqref{eq:C5} lead to
\begin{align}\nn
I_5&\le \int_0^t k_2(1+(t-s)^{-\frac{1}{2}})e^{-\lambda_1(t-s)}
\big(\nbar_0+(M_1+{k_1})\eps\big)(1+s^{-\frac{N}{2p_0}})
\eps e^{\sigma(M_1+{k_1})\eps}e^{- \al_1 s}ds\\\nn
&\le C_5k_2\Big(\nbar_0+(M_1+{k_1})\eps\Big) e^{(M_1+
{k_1})\sigma\eps}\eps\left(1+t^{-\frac{1}{2}}\right)e^{-\al_1 t}
\end{align}
for all $t\in(0,T)$ by \eqref{eq:C5}.\\
Next, 
using Lemma \ref{lem:heat} (ii) and H\"older's inequality, we derive that
\begin{align*}
I_6&\le \int_0^tk_2(1+(t-s)^{-\frac{1}{2}-\frac{N}{2q_0}})e^{-\lambda_1(t-s)}\|u (\cdot,s)\cdot\nabla c(\cdot,s)\|_{L^{q_0}(\Om)}ds\\\nn
&\le \int_0^tk_2(1+(t-s)^{-\frac{1}{2}-\frac{N}{2q_0}})e^{-\lambda_1(t-s)}
\|u(\cdot,s)\|_{L^{q_0}(\Om)}\|\nabla c(\cdot,s)\|_{\Lin}ds\fat.
\end{align*}
If we insert estimates from \eqref{eq:T} and employ Lemma \ref{lem:product} and \eqref{eq:C6}, we see that
\begin{align*}
I_6&\le \int_0^tk_2(1+(t-s)^{-\frac{1}{2}-\frac{N}{2q_0}})e^{-\lambda_1(t-s)}
M_3\eps (1+s^{-\frac{1}{2}+\frac{N}{2q_0}})e^{-\al_2s}
M_2\eps(1+s^{-\frac{1}{2}})
e^{-\al_1 s}ds\\\nn
&\le 3\int_0^tk_2(1+(t-s)^{-\frac{1}{2}-\frac{N}{2q_0}})e^{-\lambda_1(t-s)}M_3\eps (1+s^{-1+\frac{N}{2q_0}})M_2\eps e^{-\al_1 s}ds\\\nn
&\le 3k_2M_2M_3\eps^2C_6\left(1+t^{-\frac{1}{2}}\right)e^{-\al_1 t}
\end{align*}
for all $t\in(0,T)$.
Combining the above inequalities, we obtain
\begin{align}\label{eq:finalestimateinlemestnac}
\nn
{\|\nabla c(\cdot,t)\|_{\Lin}
}
&\le \Big(k_2+C_5k_2 (\nbar_0+(M_1+{k_1})\eps) e^{(M_1+{k_1})\sigma\eps}+3k_2M_2M_3\eps C_6\Big) \left(1+t^{
-\frac12}\right)e^{-\alpha_1 t} \eps\\
& \le \frac{M_2\eps}2\left(1+t^{-\frac{1}{2}}\right)e^{-\al_1 t}
\end{align}
holds for all $t\in(0,T)$ by \eqref{eq:M2}.
\end{proof}

Having achieved these estimates for $\nabla c$, we may re-examine the first solution component and sharpen the estimate from Lemma \ref{lem:prelimestimaten}.
\begin{lem}\label{lem:estimaten}
Under the assumptions of Proposition \ref{prop:Szero},
finally also
\begin{align*}
\|n(\cdot,t)-e^{t\Lap}n_0\|_{L^\theta(\Om)}<\frac{M_1\eps}2\left(1+t^{-\frac{N}2(\frac1{p_0}-\frac1\theta)}\right)e^{-\al t}
\end{align*}
is valid for all $t\in(0,T)$ and for all  $\theta\in[q_0,\infty]$.
\end{lem}
\begin{proof}
Let $\theta\in[q_0,\infty]$ and $t\in(0,T)$. Then
\begin{align}\nn
\|n(\cdot,t)-e^{t\Delta}n_0\|_{\Lt}&\le\int_0^t\|e^{(t-s)\Delta}\nabla\cdot({nS(\cdot,n,c){\cdot}\nabla c})(\cdot,s)\|_{\Lt}ds+\int_0^t\|e^{(t-s)\Delta}u(\cdot,s)\cdot\nabla n(\cdot,s)\|_{\Lt}ds\\
&=:I_7+I_8\nn
\end{align}
and according to Lemma \ref{lem:heat}(iv) we have
\begin{align}\nn
I_7&\le \int_0^tk_4(1+(t-s)^{-\frac{1}{2}-\frac{N}{2}(\frac{1}{q_0}-\frac{1}{\theta})})e^{-\lambda_1(t-s)}
\|{(n S(\cdot,n,c){\cdot}\nabla c)(\cdot,s)}\|_{L^{q_0}(\Om)}ds\\\nn
&\le C_S\int_0^tk_4(1+(t-s)^{-\frac{1}{2}-\frac{N}{2}(\frac{1}{q_0}-\frac{1}{\theta})})e^{-\lambda_1(t-s)}
\|n(\cdot,s)\|_{L^{q_0}(\Om)}\|\nabla c(\cdot,s)\|_{\Lin}ds.
\end{align}
Here we can employ the estimates provided by \eqref{eq:n_Ltheta}, \eqref{eq:T} and Lemma \ref{lem:product} to gain
\begin{align}
\nn I_7 &\le C_S\int_0^tk_4(1+(t-s)^{-\frac{1}{2}-\frac{N}{2}(\frac{1}{q_0}-\frac{1}{\theta})})e^{-\lambda_1(t-s)}
(\nbar_0|\Om|^{\frac{1}{q_0}}+(M_1+{k_1})\eps)(1+s^{-\frac{N}{2}(\frac1{p_0}-\frac{1}{q_0})})
M_2\eps(1+s^{-\frac{1}{2}})e^{-\al_1 s}ds\\\nn
&\le 3C_Sk_4M_2\left({\nbar_0|\Om|^{\frac{1}{q_0}}}+(M_1+
{k_1})\eps\right)\eps\int_0^t(1+(t-s)^{-\frac{1}{2}-\frac{N}{2}(\frac{1}{q_0}-\frac{1}{\theta})})e^{-\lambda_1(t-s)}(1+s^{-\frac12-\frac N2(\frac1{p_0}-\frac1{q_0})})e^{-\al_1 s}ds\\\nn
&\le 3C_SC_7k_4M_2\left({m|\Om|^{\frac{1}{q_0}}}+(M_1+{k_1})\eps\right)\ep
\left(1+t^{-\frac{N}{2}(\frac{1}{p_0}-\frac{1}{\theta})}\right)e^{-\al_1 t}.
\end{align}
As $\nbar_0$ is constant and $\na\cdot u=0$,
\[
 I_8=\intnt \norm[\Lt]{e^{(t-s)\Lap} {\left(u\cdot\na(n-\nbar_0)\right)(\cdot,s)}}ds=\intnt\norm[\Lt]{e^{(t-s)\Lap}\na\cdot{((n-\nbar_0)u)(\cdot,s)}}ds
\]
and hence, treating this integral similarly as $I_7$ before, we obtain
\begin{align}\nn
I_8&\le\int_0^tk_4(1+(t-s)^{-\frac{1}{2}-\frac{N}{2}(\frac{1}{q_0}-\frac{1}{\theta})})
e^{-\lambda_1(t-s)}\|(n(\cdot,s)-\bar{n}_0)u(\cdot,s)\|_{L^{q_0}(\Om)}ds\\\nn
&\le \int_0^tk_4(1+(t-s)^{-\frac{1}{2}-\frac{N}{2}(\frac{1}{q_0}-\frac{1}{\theta})})
e^{-\lambda_1(t-s)}\|n(\cdot,s)-\bar{n}_0\|_{\Lin}\|u(\cdot,s)\|_{L^{q_0}(\Om)}ds\\\nn
&\le
\int_0^tk_4(1+(t-s)^{-\frac{1}{2}-\frac{N}{2}(\frac{1}{q_0}-\frac{1}{\theta})})
e^{-\lambda_1(t-s)}(M_1+k_1)\eps(1+s^{-\frac{N}{2p_0}})e^{-\al_1 s}
M_3\eps (1+s^{-\frac{1}{2}+\frac{N}{2q_0}})e^{-\al_2s}ds
\\\nn
&\le 3(M_1+{k_1})k_4M_3\eps^2\int_0^t(1+(t-s)^{-\frac{1}{2}-\frac{N}{2}
(\frac{1}{q_0}-\frac{1}{\theta})})
e^{-\lambda_1(t-s)}
(1+s^{-\frac{1}{2}-\frac{N}{2p_0}+\frac{N}{2q_0}})e^{-\al_1 s}ds\\\nn
&\le 3(M_1+{k_1})C_7k_4M_3\eps^2\left(1+t^{-\frac{N}{2}(\frac{1}{p_0}-\frac{1}{\theta})}\right)e^{-\al_1 t}.
\end{align}
Using the choice of $\eps$ and \eqref{eq:M1} we arrive at
\begin{align*}\nn
&\|n(\cdot,t)-e^{t\Delta}n_0\|_{\Lt}\\
&\le \left(3C_SC_7k_4M_2m|\Om|^{\frac{1}{q_0}}+3C_SC_7k_4M_2(M_1+
{k_1})\eps+3(M_1+{k_1})C_7k_4M_3\eps\right)\eps\left(1+t^{-\frac{N}{2}(\frac{1}{p_0}-\frac{1}{\theta})}\right)e^{-\al_1 t}\\\nn
&\le \frac{M_1\eps}2\left(1+t^{-\frac{N}{2}(\frac{1}{p_0}-\frac{1}{\theta})}\right)e^{-\al_1 t}
\end{align*}
for all $t\in(0,T)$.
\end{proof}
While we have obtained some estimates for $u$, one for $\norm[L^2]{A^\beta u{(\cdot,t)}}$ is not yet among them, although this is the quantity featured by the extensibility criterion in Lemma \ref{lem:locexistence}. We rectify this in the next lemma:

\begin{lem}\label{lem:bounduDbeta}
Given $N$, $p_0$, $q_0$, $q_1$, $\beta$, $C_S$, $\Phi$, $m$, $\alpha_1$, $\alpha_2$, $\eps$ as in the statement of Proposition \ref{prop:Szero}, it is possible to find $C_8>0$ with the property asserted there. In particular, for any $t\in(0,T)$, we have
\begin{align}\label{eq:u_Abeta}
 \norm[L^2(\Om)]{A^\beta u(\cdot,t)}\leq { C_8 } e^{-\alpha_2 t}.
\end{align}
\end{lem}
\begin{proof}
We first define $M(t):=e^{\al_2 t}\|A^\beta u(\cdot,t)\|_{L^2(\Om)}$ for $t\in(0,T)$.
Moreover, let us pick $r>N$ such that
\[
 \frac1{q_0}+\frac1N>\frac1r\geq\frac1N+\frac12-\frac{2\beta}N,
\]
which is evidently possible due to $\frac{2\beta}N>\frac2N\cdot\frac N4=\frac12$. If we set $b:=\frac1{q_0}/(\frac1{q_0}+\frac1N-\frac1r)$, we have $b\in(0,1)$ and the Gagliardo-Nirenberg inequality and Lemma \ref{lem:stokes}(iv) provide us with $c_1>0$
and $c_2=k_9(\beta,r,2)c_1$
 such that
\[
 \norm[L^\infty(\Om)]{\Phii}\leq c_1\norm[W^{1,r}(\Om)]{\Phii}^b\norm[L^{q_0}(\Om)]{\Phii}^{1-b} \leq c_2\norm[L^2(\Om)]{A^\beta\Phii}^b\norm[L^{q_0}(\Om)]{\Phii}^{1-b} \qquad \mbox{ for all } \Phii\in L^{q_0}(\Om)\cap W^{1,r}(\Om) \cap L^2_\sigma(\Om).
\]
In particular,
\begin{align}\label{eq:estunaul2}
 \norm[L^2(\Om)]{{(u\cdot\na) u (\cdot,s)}} \leq& \norm[L^\infty(\Om)]{u(\cdot,s)} |\Om|^{\frac12-\frac1N} \norm[L^N(\Om)]{\na u(\cdot,s)}\nn \\
 \leq& c_2|\Om|^{\frac12-\frac1N} \norm[L^2(\Om)]{A^\beta u(\cdot,s)}^b \norm[L^{q_0}(\Om)]{u(\cdot,s)}^{1-b}\norm[L^N(\Om)]{\na u(\cdot,s)}, \qquad s\in(0,T).
\end{align}
We set
\begin{align*}
 t_0:=\tau(q_0,\beta,\eps,\eps,\eps,C_S), \qquad
 \Gamma:=\Gamma(q_0,\beta,\eps,\eps,\eps,C_S)
\end{align*}
as provided by Lemma \ref{lem:locexistence} and choose $c_3>0$ such that $\norm[L^{q_0}(\Om)]{\Phii}\leq c_3\norm[L^2(\Om)]{A^\beta \Phii}$ for all $\Phii\in D(A^\beta)$.
If we use that $\norm[W^{1,N}{(\Om)}]{u}\leq k_9(\beta,N,2) \norm[L^2(\Om)]{A^\beta u}$ according to Lemma \ref{lem:stokes}(iv), \eqref{eq:estunaul2} then shows that
\begin{equation} \label{eq:est_unaushort}
\norm[L^2(\Om)]{{(u\cdot\na) u (\cdot,s)}}\leq
c_2|\Om|^{\frac12-\frac1N} \Gamma^b c_3\Gamma^{1-b} \norm[L^N(\Om)]{\na u(\cdot,s)}
\leq c_2|\Om|^{\frac12-\frac1N}\Gamma^2 k_9(\beta,N,2)=:c_4
\end{equation}
for  $s\in(0,t_0)$, and that
\begin{align}\label{eq:est_unaulong}
 \norm[L^2(\Om)]{{(u\cdot\na) u (\cdot,s)}}\leq& c_2|\Om|^{\frac12-\frac1N} e^{-\al_2 bs}e^{-\al_2 (1-b)s} M^{b}(s)\left(M_3\eps(1+\left(\nicefrac{t_0}{2}\right)^{-\frac12+\frac N{2q_0}})\right)\left(M_4\eps(1+(\nicefrac{t_0}{2})^{-\frac12})\right)\nn\\
 =&c_5 e^{-\al_2 s} M^{b}(s) 
\end{align}
for all $s\in(\frac{t_0}2,T)$ for an obvious choice of $c_5>0$.
For $t>t_0$ we now aim at estimating
\begin{align}\nn
 \norm[L^2(\Om)]{A^\beta u (\cdot,t)} &\leq \norm[L^2(\Om)]{A^\beta e^{-tA} u_0} +\intnt \norm[L^2(\Om)]{A^\beta e^{-(t-s)A} \P (n(\cdot,s)-\nbar_0)\na\Phi} \\\label{eq:estimateAbetau}
 &~~~~~~~~~~~~~~~~~~~+\intnt \norm[L^2(\Om)]{A^\beta e^{-(t-s)A} \P{(u\cd\na u)}(\cdot,s)} ds
\end{align}
and observe that
\begin{equation}\label{eq:abetau0}
 \norm[L^2(\Om)]{A^\beta e^{-tA} u_0}\leq k_6(2,\beta) t^{-\beta}e^{-\mu t}\norm[L^2(\Om)]{u_0}\leq k_6(2,\beta) t_0^{-\beta}|\Om|^{\frac{N-2}{2N}}e^{-\alpha_2t}\norm[L^N(\Om)]{u_0}
\end{equation}
for {$t\in[t_0,T)$}.\\
Since $\beta\in(0,1)$ and $-\frac{N}{2}(\frac1{p_0}-\frac1{q_0})\in(-1,0)$ and $1-\beta-\frac N2(\frac1{p_0}-\frac1{q_0})>-1$, Lemma \ref{lem:integralestimates} and Lemma \ref{lem:improveintestimate} provide $c_6>0$ such that for all $t>0$

\begin{align}\label{eq:C42}
 \intnt (t-s)^{-\beta}e^{-\mu (t-s)}(1+s^{-\frac{N}{2}(\frac1{p_0}-\frac1{q_0})}) e^{-\alpha_1 s} ds \leq c_6\left(1+ t^{-1}\right)e^{-\alpha_1 t}.
\end{align}

From Lemma \ref{lem:stokes}(i), Lemma \ref{lem:helmholtzbounded}, Lemma \ref{lem:integralestimates} and \eqref{eq:T}, we infer
\begin{align}\label{eq:est_secondint}
 &\intnt\norm[L^2(\Om)]{A^\beta e^{-(t-s)A} \P (n(\cdot,s)-\nbar_0)\na\Phi}\nn\\
 &\leq k_6(2,\beta)k_5(2) \intnt e^{-{\mu}(t-s)} (t-s)^{-\beta} \norm[L^2(\Om)]{n(\cdot,s)-\nbar_0}\norm[L^\infty(\Om)]{\na\Phi}ds \nn\\
  &\leq k_6(2,\beta)k_5(2) \intnt e^{-{\mu}(t-s)} (t-s)^{-\beta} {|\Om|^{\frac12-\frac1{q_0}}} \norm[L^{q_0}(\Om)]{n(\cdot,s)-\nbar_0}\norm[L^\infty(\Om)]{\na\Phi}ds\nn\\
&\le k_6(2,\beta)k_5(2)\norm[\Lin]{\na\Phi}{|\Om|^{\frac12-\frac1{q_0}}} \intnt e^{-{\mu}(t-s)} (t-s)^{-\beta}
(M_1+{k_1})\ep(1+s^{-\frac N2(\frac1{p_0}-\frac1{q_0})})e^{-{\al_1} s}ds\nn\\
&\le k_6(2,\beta)k_5(2)\norm[\Lin]{\na\Phi}(M_1+{k_1})c_6{|\Om|^{\frac12-\frac1{q_0}}} \ep (1+t^{-1})e^{-{\al_1} t}\nn\\
&\le {k_6(2,\beta)k_5(2)\norm[\Lin]{\na\Phi}(M_1+{k_1})c_6{|\Om|^{\frac12-\frac1{q_0}}} \ep (1+t_0^{-1})e^{-\al_1 t}\qquad \mbox{for all } t\in[t_0,T).}
\end{align}

Moreover, from $-\beta\in(-1,0)$
and $0\geq \min\set{0,1-\beta{-\frac12}}{>}-1$, by means of Lemma \ref{lem:integralestimates} and Lemma \ref{lem:improveintestimate} we conclude the existence of $c_7>0$ such that
{\begin{equation}\label{eq:C51}
 \intnt (t-s)^{-\beta}e^{-\mu(t-s)}e^{-\alpha_1 s} ds \leq c_7 (1+t^{-1})e^{-\alpha_1 t}
\end{equation}}
holds for any $t>0$.
Furthermore, for any {$t\in[t_0,T)$} we have
\begin{align*}
 \intnt \norm[L^2(\Om)]{A^\beta e^{-(t-s)A} {(\P(u\cdot\na)u)(\cdot,s)}} ds&\leq k_6(2,\beta)k_5(2) \int_0^{\frac{t_0}2} (t-s)^{-\beta} e^{-\mu(t-s)} \norm[L^2(\Om)]{{(u\cd\na u)}(\cdot,s)} ds\\
 &\;\;+ k_6(2,\beta)k_5(2)\int_{\frac{t_0}2}^t (t-s)^{-\beta} e^{-\mu(t-s)} \norm[L^2(\Om)]{{(u\cd\na u)}(\cdot,s)} ds,
\end{align*}
where we can use \eqref{eq:est_unaushort} to estimate the first summand by
\begin{align}\label{eq:est_intshorttime}
 \int_0^{\frac{t_0}2} (t-s)^{-\beta} e^{-\mu(t-s)} \norm[L^2(\Om)]{{(u\cd\na u)}(\cdot,s)} ds\leq& \int_0^{\frac{t_0}2} \left(\nicefrac{t_0}2\right)^{-\beta} e^{-\mu t} e^{\mu s} c_4 ds \leq c_4 \left(\nicefrac{t_0}2\right)^{-\beta} (e^{\nicefrac{\mu t_0}{2}} -1)e^{-\alpha_2 t},
\end{align}
whereas the integral concerned with larger times by \eqref{eq:est_unaulong} can be controlled according to
\begin{align}\label{eq:est_intlargetime}
 \int_{\frac{t_0}2}^t (t-s)^{-\beta} e^{-\mu(t-s)} \norm[L^2(\Om)]{{(u\cd\na u)}(\cdot,s)} ds \nn
 &\leq \int_{\frac{t_0}2}^t (t-s)^{-\beta} e^{-\mu(t-s)} c_5 e^{-\al_2 s} {M^b(s)} ds \nn\\
  &\leq  c_5\supszerot {M^b(s)} \int_0^t (t-s)^{-\beta} e^{-\mu(t-s)} e^{-\al_2 s} ds\nn\\ &\leq c_5c_7 e^{-\alpha_2t} \supszerot {M^b(s)}
\end{align}
for all {$t\in[t_0,T)$}, due to \eqref{eq:C51}.
As to {$t\in(0,t_0)$}, we know from Lemma \ref{lem:locexistence} that
\begin{equation} \label{eq:est_shorttime}
 \norm[L^2(\Om)]{A^\beta u(\cdot,t)} \leq \Gamma \leq \Gamma e^{\al_2 t_0}e^{-\al_2 t} \mbox{ for all }t\in(0,t_0).
\end{equation}
If we then insert \eqref{eq:abetau0}, \eqref{eq:est_secondint}, \eqref{eq:est_intshorttime} and \eqref{eq:est_intlargetime} into \eqref{eq:estimateAbetau} and
take into account \eqref{eq:est_shorttime}, we obtain some $c_8>0$ such that for all $t\in(0,T)$
\[
 \norm[L^2(\Om)]{A^\beta u} \leq c_8e^{-\al_2t} + c_8 e^{ { -\al_2 } t} {\suptzeroT M^b(t)},
\]
where multiplication by $e^{\alpha_2 t}$ shows that
\[
 M(t) \leq c_8 + c_8{\suptzeroT M^b(t)} \qquad \mbox{for all } t\in(0,T)
\]
Due to $b<1$, we may hence infer the existence of $C_8>0$ such that
\[
 C_8 \geq M(t) = e^{\al_2 t} \norm[L^2(\Om)]{A^\beta u(\cdot,t)} \qquad \mbox{for all }t\in(0,T)
\]
This entails \eqref{eq:u_Abeta}.
\end{proof}

In order to infer the decay asserted in Proposition \ref{prop:Szero}, we have to combine the estimates from Definition \ref{def:1} with Lemma \ref{lem:locexistence}.
\begin{lem}\label{converge}
Given $N$, $p_0$, $q_0$, $q_1$, $\beta$, $C_S$, $\Phi$, $m$, $\alpha_1$, $\alpha_2$, $\eps$ as in the statement of Proposition \ref{prop:Szero}, it is possible to find there are $C_{9}>0$,  $C_{10}>0$ and $C_{11}>0$ with the properties asserted there. In particular,
\begin{align}
\label{ul2}&\|u(\cdot,t)\|_{L^\infty(\Om)}\le  { C_{9} } e^{-\al_2 t},\\
\label{nlinf}&\|n(\cdot,t)-\bar{n}_0\|_{L^\infty(\Om)}\le
 C_{10}  e^{-\al_1 t} \\
\quad\mbox{and}\quad&\|c(\cdot,t)\|_{W^{1,q_1}(\Om)}\le C_{11} e^{-\al_1 t}
\end{align}
for all $t\in(0,T)$.
\end{lem}
\begin{proof}
Since $D(A^\beta)\embeddedinto L^\infty(\Om)$ with $\beta\in(\frac N4,1)$, we can conclude the existence of
$C_{9}>0$ such that \eqref{ul2} holds from Lemma \ref{lem:bounduDbeta}.
If we set \begin{align*}
 t_0:=\tau(q_1,\beta,\eps,\eps,\eps,C_S), \qquad
 \Gamma:=\Gamma(q_1,\beta,\eps,\eps,\eps,C_S)
\end{align*}
as provided by Lemma \ref{lem:locexistence}, we see that Lemma \ref{lem:locexistence} ensures $\norm[L^\infty(\Om)]{n(\cdot,t)}\leq \Gamma$ on $[0,t_0)$, and thus
\[
 \norm[\Lin]{n(\cdot,t)-\nbar_0}\leq \norm[\Lin]{n(\cdot,t)}+\norm[\Lin]{\nbar_0} \leq \Gamma+ m \qquad \mbox{ for }t\in(0,t_0),
\]
that is
\[
 \norm[\Lin]{n(\cdot,t)-\nbar_0}\leq (\Gamma+m)e^{\alpha t_0} e^{-\alpha_1 t} \qquad \mbox{ for }t\in[0,t_0).
\]
At the same time, Lemma \ref{lem:prelimestimaten} asserts that
\[
 \norm[\Lin]{n(\cdot,t)-\nbar_0}\leq (M_1+{k_1})\Big(1+t^{-\frac N{2p_0}}\Big)e^{-\alpha_1 t}\leq (M_1+k_1)\Big(1+t_0^{-\frac N{2p_0}}\Big)e^{-\alpha_1 t}, \qquad \mbox{ for }t\in(t_0,T)
\]
so that with ${ C_{10} }=\max\left\{(\Gamma+m)e^{\alpha t_0}, (M_1+
k_1)\Big(1+t_0^{-\frac N{2p_0}}\Big)\right\}$, we have
\[
 \norm[\Lin]{n(\cdot,t)-\nbar_0}\leq C_{11}e^{-\alpha t} \qquad \mbox{ for all }t>0.
\]
Lemma \ref{lem:locexistence} also guarantees that $\norm[W^{1,q_1}(\Om)]{c(\cdot,t)}\leq \Gamma$, and hence $\norm[W^{1,q_1}(\Om)]{c(\cdot,t)}\leq \Gamma e^{\alpha_1 t_0} e^{-\alpha_1 t}$ for all $t\in [0,t_0)$. Combining this with Lemma \ref{lem:estimatenac} and Lemma \ref{lem:estimatec}, which show that
\[
 { \norm[\Lin]{\na c(\cdot,t)}\leq \frac{\eps M_2}2\left(1+t^{-\frac12}\right)e^{-\alpha_1 t} }, \quad \norm[\Lin]{c(\cdot,t)}\leq e^{(M_1+
 {k_1})\sigma \eps} \eps e^{-\alpha_1 t}
\]
for all $t>0$,
we can infer that
\[
 \norm[W^{1,q_1}(\Om)]{c(\cdot,t)}\leq C_{11} e^{-\alpha_{1} t}, \mbox{ for all }t>0.
\]
where $C_{11} =\max\left\{\Gamma e^{\alpha t_0},\eps M_2|\Om|^{\frac1{q_1}}\left(1+t_0^{-\frac12}\right),2|\Om|^{\frac1{q_1}} e^{(M_1+{k_1})\sigma\eps}\eps\right\}$.
\end{proof}

Now we are ready to complete the proof Proposition \ref{prop:Szero}.

\begin{proof}[Proof of Proposition \ref{prop:Szero}]
 First we claim that the solution is global. In order to show this,
 {we} observe that if $\Tmax<\infty$, then according to the blow-up criterion in \eqref{eq:blowupcrit}, the inequalities required in the definition \eqref{eq:T} of $T$, and Lemma \ref{lem:bounduDbeta}, we have $T<\Tmax$ and one of the following holds:
 \begin{align*}
&\|n(\cdot,T)- e^{T\Delta} n_0\|_{L^\theta(\Om)}=M_1\eps\left(1+T^{-\frac{N}{2}(\frac1{p_0}-\frac1{\theta})}\right)e^{-\al_1 T},\\
&  {\|\nabla c(\cdot,T)\|_{\Lin}= M_2\eps\left(1+T^{-\frac{1}{2}}\right)e^{-\al_1 T}
},\\
&\norm[L^{q_0}(\Om)]{u(\cdot,T)}= M_3 \eps \left(1+T^{-\frac12+\frac{N}{2q_0}}\right)e^{-\al_2 T},\\
&\norm[L^{N}(\Om)]{\nabla u(\cdot,T)}= M_4\eps \left(1+T^{-\frac12}\right)e^{-\al_2 T},
 \end{align*}
for some $\theta\in[q_0,\infty]$.
But these quantities continuously depend on $t$ and hence each of these items would contradict Lemma \ref{lem:estimaten}, Lemma \ref{lem:estimatenac}, Lemma \ref{lem:estimateu} or Lemma \ref{lem:estimatenau}, respectively. The same contradiction arises if $\Tmax=\infty$ and $T<\infty$.
Hence $T=\infty=\Tmax$. 
The remaining estimates and assertions about convergence result from Definition \ref{def:T} and Lemma \ref{converge}.
\end{proof}

\begin{remark}\label{rem:otherconditions}
After having shown Proposition \ref{prop:Szero}, let us briefly indicate the changes that are necessary in order to prove Theorem \ref{thm:alternative} instead of Theorem \ref{thm:main}.
Indeed, these are confined to the proof of the counterpart of Proposition \ref{prop:Szero}; the approximation procedure that is to follow in Section \ref{sec:generalS} remains unaffected.
We note that
\begin{equation}\label{eq:remainingsmallness} m=\nbar_0=\frac1{|\Om|}\int n_0\le |\Om|^{-\frac1{p_0}}\norm[L^{p_0}(\Om)]{n_0}\le |\Om|^{-\frac1{p_0}}\ep \end{equation}
 and hence, in particular, $\norm[L^{p_0}(\Om)]{\nbar_0} \le \eps$ and $\norm[L^{p_0}(\Om)]{n_0-\nbar_0}<2\eps$ so that in \eqref{eq:n_Ltheta} (and by extension, in all of Sections \ref{sec:constandparam} and \ref{sec:Thm}), replacing $k_1$ by $2k_1$ is sufficient to retain the validity of Lemma \ref{lem:prelimestimaten} and its consequences. The only remaining - but most noticable - place which is affected by the change from \eqref{init_small} to \eqref{init_small'} is Lemma \ref{lem:estimatenac}.
With the new condition, for the estimate of the first term in \eqref{eq:thisischangedbyothercond}, we invoke Lemma \ref{lem:heat}(iii) instead of Lemma \ref{lem:heat}(ii). In the estimate of $I_5$, we have to exchange a factor $\eps$ by $\norm[L^\infty(\Om)]{c_0}=M$, but can, thanks to \eqref{eq:remainingsmallness}, rely on the smallness of $(\nbar_0+(M_1+2k_1)\eps)\leq(|\Om|^{-\frac1{p_0}}+M_1+2k_1)\eps$ instead, so that \eqref{eq:finalestimateinlemestnac} would read
\begin{align*}
 \norm[L^\infty(\Om)]{\na c(\cdot,t)}&\leq\left(k_3+C_5k_2\left(|\Om|^{-\frac1{p_0}}+M_1+2k_1\right)Me^{(M_1+2k_1)\sigma\eps}\eps
 +3k_2M_2M_3C_6\eps \right)\left(1+t^{-\frac12}\right)e^{-\al_1t}\eps\\
&\leq \frac{M_2}{2}\eps\left(1+t^{-\frac12}\right)e^{-\al_1t}.
\end{align*}
Of course, this mandates changes in Lemma \ref{lem:chooseM} also. We give an appropriately modified version in the appendix (Lemma \ref{lem:chooseM_alternative}).
\end{remark}

\section{System with rotational flux (general S)}\label{sec:generalS}
In this section, we deal with the more general model, where $S\in C^2(\Ombar\times[0,\infty);\mathbb{R}^{N\times N})$ is a more arbitrary matrix-valued function, without the requirement of being zero close to the boundary.
In this case, we construct solutions by an approximation procedure.
In order to make the previous result applicable, we introduce a family of smooth functions
\begin{align}\label{rho}
\rho_\eta\in C_0^\infty(\Om) \text{ and } 0\le\rho_\eta(x)\le1 \; \mbox{ for }\eta\in(0,1), \;\;\rho_\eta(x)\nearrow 1 \text{ as } \eta\searrow 0
\end{align}
and given any function $S$ satisfying the assumptions of Theorem \ref{thm:main}, we let
\begin{equation}
 \label{eq:Seta}\Sep(x,n,c)=\rho_\eta(x)S(x,n,c).
\end{equation}
 Using this definition, we  regularize (\ref{0}) as follows:
\be{eta}
\left\{
\begin{array}{llc}
{n_\eta}_t=\Delta n_\eta-\nabla\cdot(n_\eta S_\eta (x,n_\eta,c_\eta)\cdot\nabla c_\eta)-u_\eta\cdot\nabla n_\eta, &(x,t)\in \Omega\times (0,T),\\[6pt]
\displaystyle
{c_\eta}_t=\Delta c_\eta-n_\eta c_\eta-u_\eta\cdot\nabla c_\eta, &(x,t)\in\Omega\times (0,T),\\[6pt]
\displaystyle
{u_\eta}_t=\Delta u_\eta-(u_\eta\cdot\nabla )u_\eta+\nabla P+n_\eta\nabla\Phi, \quad\nabla\cdot u_\eta=0, &(x,t)\in\Omega\times (0,T),\\[6pt]
\displaystyle
\nabla n_\eta\cdot\nu=\nabla c_\eta\cdot\nu=0,\;\;u_\eta=0,&(x,t)\in \partial\Omega\times (0,T),\\[6pt]
n_\eta(x,0)=n_{0}(x),\quad  c_\eta(x,0)=c_{0}(x),\quad  u_\eta(x,0)=u_{0}(x),
& x\in\Omega.
\end{array}
\right.
\ee

We have chosen $\Sep$ in such a way that it satisfies the additional condition imposed in Proposition \ref{prop:Szero}. Therefore the existence of solutions follows from the previous section:

\begin{lem}\label{lem:approx_existence}
{ Let $N\in\set{2,3}$,
{$p_0\in(\frac N2,{\infty})$}, $q_0\in(N,{\infty})$, and $\beta\in(\frac N4,1)$. Let $C_S>0$, $\Phi\in C^{1+\delta}(\Ombar)$ with some $\delta>0$, $m>0$. Let $\alpha_1\in(0,\min\{m,\lambda_1\})$ and $\alpha_2\in(0,\min\set{\alpha_1,\lambda_1'})$. }
 Let $(n_0,c_0,u_0)$ satisfy \eqref{inidata} and \eqref{init_small}.
 Then for any $\eta\in(0,1)$ there is a global classical solution $(\nep,\cep,\uep,P_\eta)$ of \eqref{eta} and there {are constants $C_8, C_{9}, C_{10}, C_{11}>0$} such that for any $\eta\in(0,1)$ the estimates
 \begin{equation}\label{eq:expdecayeta}
  \norm[W^{1,q_0}(\Om)]{\cep(\cdot,t)}\le {C_{11}}e^{-\al_1 t}, \quad \norm[\Lin]{\nep(\cdot,t)-\nbar_0}\le {C_{10}} e^{-\al_1 t}, \quad \norm[\Lin]{\uep(\cdot,t)}\le {C_{9}}e^{-\al_2 t}
 \end{equation}
 hold for any $t>0$ and such that moreover the solutions satisfy
 \begin{equation}\label{eq:decayuetaDbeta}
  \norm[L^2(\Om)]{A^\beta \uep(\cdot,t)}\le {C_8}e^{-\al_2t}
 \end{equation}
 for any $t>0$ and any $\eta\in(0,1)$.
 { Moreover there is $C_{12}>0$ such that for any $\eta\in(0,1)$ and any $t>0$
 \begin{equation}\label{eq:decaynaceta}
  \norm[\Lin]{\na c(\cdot,t)}\leq {C_{12}}\left(1+t^{-\frac12}\right)e^{-\al_1 t}.
 \end{equation}}
\end{lem}
\begin{proof}
 These assertions are part of Proposition \ref{prop:Szero} if we set $C_{13}:=\eps M_2$ in \eqref{eq:estimateslikeindefT}, at least for $p_0<N$, $q_0<\left(\frac1{p_0}-\frac1N\right)^{-1}$. For larger values of $p_0$ or $q_0$, \eqref{init_small} entails the validity of \eqref{init_small} for smaller $p_0$, $q_0$ if $\eps$ is adequately adjusted, and Lemma \ref{lem:approx_existence} still follows from Proposition \ref{prop:Szero}, if $q_0$, $p_0$ and $q_1$ are suitably chosen therein.
\end{proof}

From this family of approximate solutions we aim to extract a convergent sequence.
Already the {frail manner of} convergence of $\Sep$, however, {puts us far from} the immediate conclusion that the limiting object satisfies \eqref{0} in a pointwise sense. Accordingly, we will first ensure that it is a weak solution; afterwards we will show that it is sufficiently regular so as to be a classical solution. For this purpose, we require a definition of ``weak solution":

\begin{definition}\label{def:1}
We say that $(n,c,u)$ is a weak solution of (\ref{0}) associated to initial data $(n_0,c_0,u_0)$ if
\begin{align*}
  n,c \in L^2_{loc}([0,\infty),W^{1,2}(\Om)),
 u \in L^2_{loc}([0,\infty),W_{0,\sigma}^{1,2}(\Om)),
\end{align*}
and for all $\psi\in C_0^\infty(\Ombar\times[0,\infty))$ and all $\Psi\in C_{0,\sigma}^\infty(\Om\times[0,\infty))$ the following identities hold:
\begin{align}\label{weaksol}
-\int_0^\infty\!\!\intO n\psi_t-\intO n_0\psi(\cdot,0)&=-\int_0^\infty\!\!\intO \nabla n\cdot\nabla\psi
+\int_0^\infty\!\!\intO nS(x,n,c)\nabla c\cdot\nabla\psi+\int_0^\infty\!\!\intO nu\cdot\nabla\psi,\\
-\int_0^\infty\!\!\intO c\psi_t-\intO c_0\psi(\cdot,0)&=-\int_0^\infty\!\!\intO \nabla c\cdot\nabla\psi
-\int_0^\infty\!\!\intO nc\psi+\int_0^\infty\!\!\intO cu\cdot\nabla\psi,\\
-\int_0^\infty\!\!\intO u\cdot\Psi_t-\intO u_0\cdot\Psi(\cdot,0)&=-\int_0^\infty\!\!\intO \nabla u\cdot\nabla\Psi
+\int_0^\infty\!\!\intO (u\cdot\nabla) u\cdot{\Psi}+\int_0^\infty\!\!\intO n\nabla \Phi\cdot\Psi.
\end{align}
\end{definition}

Within this framework, we shall show the sequence of solutions to (\ref{eta}) to have a limit. We begin the extraction of convergent subsequences with convergence of $n$ and $c$ in H\"older spaces in the following lemma:

\begin{lem}\label{lem:convnc}
 There are $\gamma>0$, a sequence $\{\eta_j\}_{j{\in\N}}\searrow 0$ and $n,c\in C^{1+\gamma,\gamma}_{loc}(\Ombar\times(0,\infty))$ such that
 \begin{align}
  n_{\eta_j}\to & n\qquad\mbox{ in } C^{\gamma,\frac\gamma2}_{loc}(\Ombar\times(0,\infty))\label{con:n}\\
  c_{\eta_j}\to & c\qquad\mbox{ in } C^{\gamma,\frac\gamma2}_{loc}(\Ombar\times(0,\infty))  \label{con:c}
 \end{align}
 as $j\to\infty$.
\end{lem}
\begin{proof}
For any $\eta\in(0,1)$ the function $\nep$ is a bounded distributional solution of the parabolic equation
\[
 \ntilde_t-div\ a(x,t,\ntilde,\na \ntilde)=b(x,t,\ntilde,\na \ntilde) \qquad \mbox{in }\Om\times(0,\infty)
\]
for the unknown function $\ntilde$, with $a(x,t,\ntilde,\na \ntilde)=\na \ntilde-\nep S_\eta \na \cep -\uep \nep$, and $b\equiv 0$, {and} $a(x,t,\ntilde,\na \ntilde)\cdot \nu=0$ on the boundary of the domain.
Defining $\psi_0(x,t)=|\nep(x,t) S_\eta(x,\nep(x,t),\cep(x,t))\na\cep(x,t)|^2+|\uep(x,t)\nep(x,t)|^2$ and $\psi_1=|\nep S_\eta(\cdot,\nep,\cep)\na\cep|$ we see that $a(x,t,\ntilde,\na \ntilde)\na \ntilde\geq \frac12|\na \ntilde|^2-\psi_0$ and $|a(x,t,\ntilde,\na \ntilde)|\leq |\na \ntilde|+\psi_1$.
If we let $T>0$ and $\tau\in(0,T)$, the regularity result \cite[Thm 1.3]{porzio_vespri} therefore asserts the existence of $\gamma_1 \in(0,1)$ and $c_1>1$ such that $\norm[C^{\gamma_1,\frac{\gamma_1}2}(\Ombar\times(\tau,T))]{\nep}\leq c_1$.

According to the aforementioned theorem, these numbers $\gamma_1$ and {$c_1$} depend on $\norm[L^\infty(\Om\times(\tau,T))]{\nep}$ and the norms of $\psi_0$, $\psi_1$ in certain spaces $L^p((\tau,T), L^q(\Om))$, where $p$ and $q$ must be sufficiently large, but need not be infinite. Such bounds have been asserted independently of $\eta$ in \eqref{eq:expdecayeta} and \eqref{eq:decaynaceta} in Lemma \ref{lem:approx_existence}, so that we can conclude the existence of $\gamma_1\in(0,1)$ and $c_1>0$ such that
\[
 \normm{C^{\gamma_1,\frac{\gamma_1}2}(\Ombar\times[\tau,T])}{n_\eta}\leq c_1 \qquad \mbox{ for every }\eta\in(0,1).
\]
Moreover, since $b\equiv 0$, according to \cite[Remark 1.3]{porzio_vespri}, $\gamma_1$ is independent of $\tau$.
By a similar reasoning applied to the second equation and again invoking \cite[Thm 1.3]{porzio_vespri}, we can find $\gamma_2\in(0,1)$ and $c_2>0$ such that
\[
 \normm{C^{\gamma_2,\frac{\gamma_2}2}(\Ombar\times[\tau,T])}{\cep}\leq c_2 \qquad \mbox{ for every }\eta\in(0,1).
\]
If we now pick $\gamma\in(0,\min\set{\gamma_1,\gamma_2})$, the compact embeddings $C^{\gamma_i,\frac{\gamma_i}2}(\Ombar\times[\tau,T])\embeddedinto\embeddedinto C^{\gamma,\frac{\gamma}2}(\Ombar\times[\tau,T])$, $i\in\set{1,2}$, 
allow for extraction of a sequence such that \eqref{con:n} and \eqref{con:c} hold.
\end{proof}

In order to achieve convergence in the third component of the solutions,
we will combine estimates we already have obtained with Theorem 2.8 of \cite{giga_sohr} and the embedding result \cite[Thm 1.1]{amann}, which asserts that for {$\gamma\in(0,1)$} the set of functions with $\norm[L^p(0,T;W^{2,p}(\Om))]{u}$ and $\norm[L^p(0,T;L^p(\Om))]{u_t}$ being bounded is a compact subset of $C^{{\gamma}}(0,T;C^{1+{\gamma}}(\Ombar))$ if $p$ is large. The latter is an argument employed also in \cite[Cor. 7.7]{wk_ns_oxytaxis}, the former also lies at the center of the proof of \cite[Lemma 7.6]{wk_ns_oxytaxis}, but is substantially easier here due to the estimates stated in Lemma \ref{lem:approx_existence}.

\begin{lem}\label{lem:convu}
 There are $\gamma>0$, a subsequence $\{\eta_j\}_{j{\in\N}}\searrow 0$ of the sequence given in Lemma \ref{lem:convnc} and $u\in C^{1+\gamma,\gamma}_{loc}(\Ombar\times(0,\infty);\R^N)$ such that
 \begin{equation}\label{con:u}
  u_{\eta_j}\to u\qquad\mbox{ in } C^{1+\gamma,\gamma}_{loc}(\Ombar\times(0,\infty))
 \end{equation}
 as $j\to\infty$.
\end{lem}
\begin{proof}
Let us fix $\tau \in(0,\infty)$. We introduce a smooth, nondecreasing function $\xi\colon\R\to\R$ which satisfies $\xi(t)=0$ for $t\le \tau$ and $\xi(t)=1$ for $t\ge 2\tau$ and will consider the functions $\xi\uep$ with $\eta\in(0,1)$ in the following.
Given $s\in(1,\infty)$, \cite[Thm. 2.8]{giga_sohr} provides ${c_1=c_1}(s,\Om)$ such that, for any $\eta\in(0,1)$, $\xi \uep$, being a solution of the Stokes equation with right-hand side $\calP (\xi (\uep\cdot\na)\uep)+\calP (\xi \nep\na\Phi)+\calP (\xi' \uep)$ satisfies
\[
 \int_{{\tau}}^T \norm[L^s(\Om)]{(\xi \uep)_t}^s+\int_{{\tau}}^T\norm[L^s(\Om)]{D^2(\xi \uep)}^s \leq {c_1} \left(0+ \int_{{\tau}}^T \norm[L^s(\Om)]{\calP  (\xi \uep\cdot\na) \uep + \calP  \xi (\nep-\nbar_0)\na\Phi + \calP  \xi' \uep}^s \right)
\]
for any $T>\tau$.
From the exponential decay of $\norm[L^\infty(\Om)]{\nep-\nbar_0}$ and of $\norm[L^\infty(\Om)]{\uep(\cdot,t)}$ as stated in \eqref{eq:expdecayeta} we obtain the existence of $c_2, c_3>0$ such that for any $\eta\in(0,1)$
\begin{equation}\label{eq:estimateforu_stokesreg}
{ \int_\tau^{{T}} }
 \norm[L^s(\Om)]{(\xi \uep)_t}^s+{ \int_\tau^{{T}} }\norm[L^s(\Om)]{D^2(\xi \uep)}^s \leq c_2+c_3
 { \int_\tau^{{T}} }\norm[L^s(\Om)]{\na (\xi \uep)}^s {\qquad  \mbox{for any } T>\tau}.
\end{equation}
Let $s>N$ and fix {$r\in(1,s)$, so} that $\frac1N+\frac1r-\frac1s>0$. Defining
\[
 a=\frac{\frac1N+\frac1r-\frac1s}{\frac2N+\frac1r-\frac1s},
\]
we then observe that $a\in ({\frac12},1)$ and hence the Gagliardo-Nirenberg inequality yields a constant $c_4$ such that
\[
 \norm[L^s(\Om)]{\na (\xi \uep)(\cdot,t)}^s \leq c_4\norm[L^s(\Om)]{D^2 (\xi \uep)(\cdot,t)}^{as}\norm[L^r(\Om)]{(\xi \uep)(\cdot,t)}^{(1-a)s} \qquad \mbox{for all } t\in(0,T)
\]
and an application of {this together with} the $L^\infty$-estimate for $u$ {from \eqref{eq:expdecayeta}} and H\"older's inequality in \eqref{eq:estimateforu_stokesreg} show that there is $c_5>0$ such that for any $T>\tau$ and any $\eta\in(0,1)$
\[
 \int_\tau^{{T}}  \norm[L^s(\Om)]{{(\xi}u_{\eta})_t}^s+
 \int_\tau^{{T}}  \norm[L^s(\Om)]{D^2{(\xi}\uep)}^s
  \leq c_2+c_5{|T-\tau|^{1-a}}\left(
 \int_\tau^{{T}}  \norm[L^s(\Om)]{D^2 {(\xi}\uep)}^s\right)^a,
\]
 and we can conclude boundedness of $\norm[L^s(\tau,T;L^s(\Om))]{D^2{(\xi}\uep)}$ and then of $\norm[L^s(\tau,T;L^s(\Om))]{({\xi }u_{\eta})_t}$ with bounds independent of $\eta$.

All in all, for any $s>1$ {and any $T>2\tau$, there is ${c_6}>0$ such that} for any $t\in(2\tau,T)$ and any $\eta\in(0,1)$
\begin{equation}\label{eq:boundutD2u}
 \norm[L^s((t,T);L^s(\Om))]{u_{\eta t}}+\norm[L^s((t,{T});W^{2,s}(\Om))]{\uep}\leq {c_6}.
\end{equation}

Now, letting $\gamma'\in(0,1)$, using appropriately large $s$ and referring to \cite[Thm 1.1]{amann}, {for any $T>0$} we obtain a constant $c_{7}>0$ so that
\[
 \norm[C^{1+\gamma',\gamma'}(\Ombar\times(t,{T}))]{\uep}=\norm[C^{1+\gamma',\gamma'}(\Ombar\times(t,{T}))]{\xi \uep}\leq c_{7} \qquad \mbox{ for all }t\in(2\tau,T).
\]

Therefore, for any $\tau>0$, $T>2\tau$ we can find a subsequence of the sequence from Lemma \ref{lem:convnc} such that  $\uep\to  u$ and $\na\uep\to\na u$ in $C^{\gamma,\gamma}(\Ombar\times(t,{T}))$ for some $\gamma < \gamma'$ and for any {$t\in({2}\tau,T)$}.
\end{proof}

For $u$, this lemma already covers the convergence of first spatial derivatives. Also concerning $n$ and $c$, at least some kind of convergence of these quantities seems desirable. For the fluid velocity field, in fact, slightly higher derivatives are of interest. We obtain convergence for these in the following lemma:

\begin{lem}\label{lem:convderivatives}
 There exists a subsequence $\{\eta_j\}_{j{\in\N}}\searrow 0$ of the sequence from Lemma \ref{lem:convu} such that
\begin{align}
\label{con:nacwstar}
 \nabla\cep\wstarto&\ \nabla c &&\text{ in }L^\infty((0,\infty),L^{{q_0}}(\Om)),\\
\label{con:nablac}
 \nabla\cep\rightharpoonup&\ \nabla c &&\text{ in }L^2(\Om\times(0,\infty)),\\
\label{con:uDbeta}
 \uep \wstarto&\ u &&\text{ in } L^\infty((0,\infty),D(A^\beta)),\\
 \label{con:nablan}
 \nabla\nep\rightharpoonup&\ \nabla n &&\text{ in }L^2(\Om\times(0,\infty)),\\
\label{con:nSnablac}
 \nep S(\cdot,\nep,\cep)\nabla \cep\rightharpoonup&\ nS(\cdot,n,c)\nabla c  &&
 \mbox{ in } L^1_{loc}(\Om\times(0,\infty)),\\
 \label{con:nt}
 n_{\eta t}\rightharpoonup&\ n_t &&\mbox{ in } L^2((0,\infty),(W_0^{1,2}(\Om))^*),\\
 \label{con:ct}
 c_{\eta t}\rightharpoonup&\ c_t &&\mbox{ in } L^2((0,\infty),(W_0^{1,2}(\Om))^*),\\
 \label{con:ut}
 u_{\eta t}\rightharpoonup&\ u_t && \mbox{ in } L^2((0,\infty),(W_{0,\sigma}^{1,2}(\Om))^*).
\end{align}
as $\eta=\eta_j\searrow0$.
\end{lem}
\begin{proof}
From \eqref{eq:expdecayeta} we know that there is $c_1>0$ such that for all $\eta\in(0,1)$
\[
 \norm[L^\infty((0,\infty),L^{q_0}(\Om))]{\na\cep}\leq {c_1}.
\]
Therefore we may conclude the existence of a sequence satisfying \eqref{con:nacwstar}; this also entails \eqref{con:nablac}.\\
By the same reasoning
we can use the $\eta-$independent bound on $\norm[L^\infty((0,\infty),D(A^\beta))]{\uep}$ given by \eqref{eq:decayuetaDbeta} to extract a subsequence satisfying \eqref{con:uDbeta}.\\
Concering convergence of $\na\nep$, we multiply the first equation {of \eqref{eta}} by $\nep$ so as to obtain
\begin{align}\nn
\frac{1}{2}\frac{d}{dt}\intO n_\eta^2+\intO |\nabla n_\eta|^2 &=\intO n_\eta S_\eta{\nabla c_\eta\cdot\nabla n_\eta}
\le \frac{1}{2}\intO|\nabla n_\eta|^2+\frac12\norm[L^\infty((0,\infty)\times\Om)]{n_\eta}^2C_S^2\intO|\nabla c_\eta|^2.
\end{align}
for any $\eta\in(0,1)$ and on the whole time-interval $(0,\infty)$. Integrating this with respect to time and taking into account the exponential bound on $\io|\na\cep|^2$ and the uniform $L^\infty$-bound on $\nep$ from \eqref{eq:expdecayeta}
, we establish that
\begin{equation}
 \label{est:nablan}
 \sup_{\eta\in\{\eta_j\}} \int_0^\infty\intO |\nabla n_\eta|^2<\infty
\end{equation}
and hence can find a subsequence of the previously extracted sequence $\{\eta_j\}_{j\in\N}$ along which \eqref{con:nablan} holds.\\
Because by Lemma \ref{lem:convnc}, $\nep\to n$ and $\Sep(\cdot,\nep,\cep)\to S(\cdot,n,c)$ pointwise and $\nep$ and $\Sep(\cdot,\nep,\cep)$ both are bounded uniformly in $\eta$ due to \eqref{eq:expdecayeta} 
and \eqref{eq:condS} combined with \eqref{rho}
, from Lebesgue's dominated convergence theorem we conclude that $\nep\Sep(\cdot,\nep,\cep)\to nS(\cdot,n,c)$ in $L^2_{loc}(\Om\times(0,\infty))$. Combined with \eqref{con:nablac}, this gives \eqref{con:nSnablac}.
Turning our attention to the time derivatives, we let $\psi\in C_0^\infty(\Om)$ with $\norm[W^{1,2}(\Om)]{\psi}\leq 1$ and test the first equation of \eqref{eta} with $\psi$. We obtain
\begin{align}\nn
\left\lvert \intO (n_\eta)_t\psi\right\rvert &=\left\lvert -\intO \nabla\nep\cdot\nabla \psi+\intO\nep\Sep\nabla\cep\cdot\nabla\psi+\intO\nep\uep\cdot\nabla\psi\right\rvert\\\nn
&\le \left(\left(\intO |\nabla\nep|^2\right)^\frac12+\norm[\Lin]{\nep}C_S\left(\intO |\nabla \cep|^2\right)^\frac12
+\norm[\Lin]{\nep}\left(\intO |u|^2\right)^\frac12\right)\left(\intO |\nabla\psi|^2\right)^\frac12
\end{align}
for all $t\in(0,\infty)$, $\eta\in(0,1)$. From the definition of the norm in dual spaces and Young's inequality, we derive that
\[
 \int_0^\infty \norm[(W_0^{1,2}(\Om))^*]{n_{\eta t}}^2\leq {3}\int_0^\infty \io |\na\nep|^2+{3\norm[L^\infty(\Om\times(0,\infty))]{\nep}^2C_S^2\int_0^\infty\io |\na\cep|^2+ 3} \norm[L^\infty(\Om\times(0,\infty))]{\nep}{^2}\int_0^\infty\io|\uep|^2
\]
for all $\eta\in(0,1)$. Taking into account \eqref{est:nablan} and \eqref{eq:expdecayeta}
, we thus obtain {$c_2>0$} such that
\[
 \norm[L^2((0,\infty),(W_0^{1,2}(\Om))^*)]{n_{\eta t}}\leq c_{{2}}, \qquad \mbox{ for all }\eta\in(0,1),
\]
and may extract a further subsequence such that \eqref{con:nt} holds. The same reasoning applied to the second equation of \eqref{eta} leads to \eqref{con:ct}.
As to the third equation, employing \eqref{eq:expdecayeta} and \eqref{eq:decayuetaDbeta} and repeating the procedure with some $\psi\in C_0^\infty(\Om)$, we easily obtain uniform boundedness of $\int_0^\infty \norm[(W^{1,2}_{0,\sigma}(\Om))^*]{u_{\eta t}}^2$ (where $W^{1,2}_{\sigma,0}(\Om)= \overline{C_{0,\sigma}^{{\infty}}(\Om)}^{\norm[W^{1,2}(\Om)]{\cdot}}$) and may conclude \eqref{con:ut} along a subsequence.
\end{proof}

\begin{lem}\label{lem:isweaksol}
 The functions $n, c, u$ from Lemma \ref{lem:convnc} and Lemma \ref{lem:convu} form a weak solution to \eqref{0} in the sense of Definition \ref{def:1}.
\end{lem}
\begin{proof}
 The convergence properties exhibited in \eqref{con:n}, \eqref{con:nablan}, \eqref{con:nSnablac}, \eqref{con:u}, \eqref{con:c} and \eqref{con:nablac}  enable us to pass to the limit in the integral identities \eqref{weaksol} for $(\nep,\cep,\uep)$ {for any $\varphi\in C_0^\infty( \Ombar\times[0,\infty))$}
\end{proof}

Moreover, these weak solutions obey the desired decay estimates.
\begin{lem}\label{lem:decay}

{With $C_8, C_9, C_{10}$ and $C_{11}$ as in Lemma \ref{lem:approx_existence},} the functions $n,c,u$ obtained from Lemma \ref{lem:convnc} and Lemma \ref{lem:convu} obey the estimates
\begin{align}
\label{eq:decayest_c}
\norm[W^{1,q_0}(\Om)]{c(\cdot,t)}&\leq {2C_{11}}e^{-\al_1 t}, &\qquad&\mbox{for almost every }t>0,\\
\label{eq:decayest_n}
\norm[L^\infty(\Om)]{n(\cdot,t){-\nbar_0}} &\leq {C_{10}}e^{-\al_1 t}, &\qquad &\mbox{for every } t>0,\\
\label{eq:decayest_u}
\norm[L^\infty(\Om)]{u(\cdot,t)}&\leq
{C_{9}}e^{-\al_2t}, &\qquad &\mbox{for every }t>0,\\
\label{eq:decayest_uDbeta}
\norm[D(A^\beta)]{u(\cdot,t)}&\leq {C_8} e^{-\al_2t}, &\qquad &\mbox{for almost every } t>0.
\end{align}
\end{lem}
\begin{proof}
 The estimates \eqref{eq:decayest_n}, \eqref{eq:decayest_u} and a corresponding estimate for $\norm[L^\infty(\Om)]{c(\cdot,t)}$ result from \eqref{eq:expdecayeta} 
 and the pointwise convergence entailed by Lemma \ref{lem:convnc} and Lemma \ref{lem:convu}.
 For $t>0$ we let $\chi_{[t,\infty)}$ denote the characteristic function of the interval $[t,\infty)$ and observe that due to \eqref{con:nacwstar} also $\chi_{[t,\infty)}\na\cep \wstarto \chi_{[t,\infty)}\na c$ in $L^\infty((0,\infty),L^{q_0}(\Om))$ as $\eta=\eta_j\searrow 0$, and therefore
\[
 \normm{L^\infty([t,\infty),L^{q_0}(\Om))}{\na c}=\norm[L^\infty((0,\infty),L^{q_0}(\Om))]{\chi_{[t,\infty)}\na c}\leq \liminf_{j\to\infty} \norm[L^\infty((0,\infty),L^{q_0}(\Om)]{\chi_{[t,\infty)}\na\cep}\leq C_{12}e^{-\al t}
\]
for all $t>0$, so that \eqref{eq:decayest_c} results.
The estimate \eqref{eq:decayest_uDbeta} follows from \eqref{con:uDbeta} and \eqref{eq:decayuetaDbeta} by the same reasoning.
\end{proof}

Naturally, in our search for classical solutions we are much more interested in obtaining smoothness of higher order than in these boundedness assertions.

\begin{lem}\label{lem:c2al}
 The functions $n,c,u$ from the previous lemmata satisfy
 \begin{align}
  n\in& C^{2+{ \gamma},1+\frac{ \gamma}2}_{loc}(\Ombar\times(0,\infty))\label{smooth:n},\\
  c\in& C^{2+{ \gamma},1+\frac{ \gamma}2}_{loc}(\Ombar\times(0,\infty))\label{smooth:c},\\
  u\in& C^{2+{ \gamma},1+\frac{ \gamma}2}_{loc}(\Ombar\times(0,\infty))\label{smooth:u},
 \end{align}
 for some ${ \gamma}>0$.
\end{lem}
\begin{proof}
We fix $\tau>0$ and $T>3\tau$. Moreover we choose a smooth function $\xi\colon \R\to [0,1]$ such that $\xi(t)=0$ for $t\le 2\tau$ and $\xi(t)=1$ for all $t\ge 3\tau$.
Then we consider the problem
\[
 \begin{cases}
  \calL w=w_t - \Delta w + u\cdot \na w = -\xi c + \xi n +\xi_t c =: f& \mbox{on } (\tau,T)\\
  (\xi c)(\cdot,\tau)=0, \qquad \delny (\xi c)\big\rvert_{\dOm} =0
 \end{cases}
\]
of which clearly $w=\xi c$ is a weak solution. The coefficients of the parabolic operator $\calL$ are H\"older-continuous in $\Ombar\times[\tau,T]$ by Lemma \ref{lem:convu} and so is $f$ (by Lemma \ref{lem:convnc}).
If combined with the uniqueness result for weak solutions in \cite[Thm. III.5.1]{LSU}, Theorem IV.5.3 of \cite{LSU} therefore asserts that $\xi c\in C^{2+{ \gamma_1},1+\frac{ \gamma_1}2}(\Ombar\times[\tau,T])$ {for some $\gamma_1>0$} and we conclude that $c\in C^{2+{ \gamma_1},1+\frac{ \gamma_1}2}(\Ombar\times[3\tau,T])$ and finally \eqref{smooth:c}.\\ When attempting to apply the same theorem to $n$ (or $\xi n$, similar as before), however, we face the additional difficulty that it requires $C^{1+{ \gamma},\frac{1+{ \gamma}}2}$-regularity of the boundary values, whereas at this point we cannot guarantee more than $C^{{ \gamma},\frac{ \gamma}2}$-regularity because of the involvement of $n$ in the argument of $S$ in the boundary condition.
We apply \eqref{con:nablan} and \eqref{eq:decayest_n} 
to see that $n$ has the regularity properties needed for an application of \cite[Thm. 1.1]{Lieberman_gradient}, which
then guarantees that $n\in C^{1+{ \gamma_2},\frac{1+{ \gamma_2}}2}(\Ombar\times(0,T))$ {for some $\gamma_2>0$} and with that we can use \cite[Thm. IV.5.3]{LSU} in the same way as before and conclude \eqref{smooth:n}.\\
Turning our attention to the function $u$ we observe that $\xi {(u\cdot\nabla)} u+ \xi n\nabla\Phi + \xi'u \in C^{{ \gamma_3},\frac{ \gamma_3}2}(\Ombar\times(0,T))$ { for some $\gamma_3>0$ by Lemma \ref{lem:convu} and \eqref{smooth:n}} and hence the same holds true for $\calP  (\xi {(u\cdot\nabla) u}+ \xi n\nabla\Phi + \xi' u)$ by Lemma \ref{lem:projection_smoothness}. Therefore the Schauder estimates for Stokes' equation given in \cite[Thm. 1.1]{Solonnikov}{, if combined with the uniqueness result in \cite[Thm. V.1.5.1]{sohr_book},} assert that $\xi u$, being a solution to $(\xi u)_t = \Delta (\xi u)+\calP  [\xi (u{\cdot}\na)u + \xi n\na\Phi + \xi'u]$, $\nabla\cdot(\xi u)=0$, belongs to the space $C^{2+\gamma_3,1+\frac{\gamma_3}2}(\Ombar\times(0,T))$ for some $\gamma_3>0$ and hence $u\in C^{2+\gamma_3,1+\frac{\gamma_3}2}(\Ombar\times[3\tau,T])$, so that we finally arrive at \eqref{smooth:u}.
\end{proof}

Having obtained this smoothness, we can quickly fill in the missing information to see that $n, c, u$ are as regular as required of classical solutions.

\begin{lem}\label{lem:regularity}
 The functions $n,c,u$ satisfy
\begin{equation}\label{eq:regstatement}
\left\{
\begin{array}{llc}
n\in C^0(\Ombar\times[0,T))\cap C^{2,1}(\Ombar\times(0,T)),\\[6pt]
c\in C^0(\Ombar\times[0,T))\cap L^\infty((0,T);W^{1,q_0}(\Om))\cap C^{2,1}(\Ombar\times(0,T)),\\[6pt]
u\in C^0(\Ombar\times[0,T))\cap L^\infty((0,T);D(A^\beta))\cap C^{2,1}(\Ombar\times(0,T)).
\end{array}
\right.
\end{equation}
\end{lem}
\begin{proof}
 For each of the functions, $C^{2,1}$-regularity follows from Lemma \ref{lem:c2al}. That 
$c\in L^\infty((0,\infty),W^{1,q_0}(\Om))$ and $u\in L^\infty((0,\infty),D(A^\beta))$ is asserted by 
 \eqref{eq:decayest_c} and \eqref{eq:decayest_u}, respectively. \\
 Therefore we are left with the task of proving the continuity at $t=0$.
From \eqref{eq:decayest_c} and \eqref{con:ct} we know that for $T>0$ we have $c\in L^\infty((0,T),{W^{1,q_0}(\Om)})$ and $c_t \in L^2((0,T),(W_0^{1,2}(\Om))^*)$, where $W^{1,q_0}(\Om)\embeddedinto\embeddedinto C^0(\Ombar)\embeddedinto (W^{1,2}_0(\Om))^*$, so that a well-known embedding result (see e.g. \cite[Cor. 8.4]{Simon}) assures us that $c\in C^0(\Ombar\times [0,T])$.
For $u$ we observe that $D(A^\beta)\embeddedinto\embeddedinto C^0(\Ombar)$ and \eqref{con:ut} and \eqref{eq:decayest_uDbeta} once more make \cite[Cor. 8.4]{Simon} applicable. In order to show continuity of $n$ at $t=0$, we note that according to \eqref{eq:expdecayeta}, there is $c_1>0$ such that $\norm[L^{q_0}(\Om)]{\nep S(\cdot,\nep,\cep)\na\cep - \uep\nep}\leq c_1$ for any $\eta\in(0,1)$ and any $t>0$. Consequently, for any $\eta\in(0,1)$ and any $t>0$, we have
\begin{align*}
 &\norm[L^\infty(\Om)]{\nep(\cdot,t)-e^{t\Delta}n_0} \\
&\quad \leq \intnt \norm[\Lin]{e^{(t-s)\Delta}\nabla\cdot\left(\nep S(\cdot, \nep(\cdot,s),\cep(\cdot,s))\na\cep(\cdot,s)+\nep(\cdot,s)\uep(\cdot,s)\right)} ds \\
 &\quad \leq \intnt k_4(1+(t-s)^{-\frac12-\frac{N}{2q_0}})e^{-\lambda_1(t-s)} \norm[L^{q_0}(\Om)]{\nep(\cdot,s)S(\cdot,\nep(\cdot,s),\cep(\cdot,s))\na\cep(\cdot,s)+\nep(\cdot,s)\uep(\cdot,s)} ds\\
 &\quad \leq c_1k_4 \left(t+\intnt s^{-\frac12-\frac{N}{2q_0}} ds\right).
\end{align*}
Given $\zeta>0$ we then fix $\delta>0$ such that $\norm[L^\infty(\Om)]{e^{t\Delta} n_0 - n_0}\leq \frac\zeta3$ and $t+\int_0^t s^{-\frac12-\frac N{2q_0}}ds<\frac\zeta {3c_1k_4}$ for all $t\in(0,\delta)$. Then using the uniform convergence $n_{\eta_j}(\cdot,t)\to n(\cdot,t)$ as $j\to\infty$ asserted by Lemma \ref{lem:convnc} we pick $\eta_j$ such that $\norm[L^\infty(\Om)]{{n}(\cdot,t)-n_{\eta_j}(\cdot,t)}\leq \frac\zeta3$. Then
\[
 \norm[\Lin]{n(\cdot,t)-n_0}\leq \norm[\Lin]{n(\cdot,t)-\nep(\cdot,t)} + \norm[\Lin]{\nep(\cdot,t)-e^{t\Delta}n_0}+\norm[\Lin]{e^{t\Delta} n_0 - n_0}<\zeta
\]
for all $t\in(0,\delta)$. Thus the proof is complete.
\end{proof}

In order to prove Theorem \ref{thm:main}, we now only have to collect the results prepared during this section:

\begin{proof}[Proof of Theorem \ref{thm:main}]
 Approximating $S$ by functions $\Sep$ as indicated in \eqref{eq:Seta}, Proposition \ref{prop:Szero} has ensured the existence of solutions $(\nep,\cep,\uep,P_\eta)$ with the properties asserted in Lemma \ref{lem:approx_existence}. From the family of these approximate solutions, in Lemma \ref{lem:convnc}, Lemma \ref{lem:convu} and Lemma \ref{lem:convderivatives} we were able to extract a subsequence that converges to functions $(n,c,u)$ in a suitable sense, which according to Lemma \ref{lem:isweaksol} form a global weak solution to \eqref{0} in the sense of Definition \ref{def:1}, according to Lemma \ref{lem:regularity} have all regularity properties required of a classical solution and by Lemma \ref{lem:decay} exhibits the desired decay properties. The missing component $P$ can be obtained from \cite[Thm. V.1.8.1]{sohr_book}. In light of the smoothness of $u$, $n$, $\Phi$, the third equation of \eqref{0} asserts that $\nabla P\in C^0(\Ombar\times(0,T))$.
\end{proof}

\renewcommand\thesection{\Alph{section}}
\setcounter{section}{0}
\section{Appendix}\label{appendix}
We have postponed the proof of Lemma \ref{lem:integralestimates}, which mainly consists in elementary calculus, but is too central to the reasoning of the present work to be left unproven. We begin the Appendix by giving this proof. After that, we will take care of a result on the Helmholtz projection, which was used as tool in the proof of Lemma \ref{lem:c2al}. Finally, this appendix contains a variant of Lemma \ref{lem:chooseM} adapted to the needs of the proof of Theorem \ref{thm:alternative}.

\begin{proof}[Proof of Lemma \ref{lem:integralestimates}]
The assertion can be proven similar as in \cite[Lemma 1.2]{W4}. A simple observation shows that
\begin{align}\label{eq:lemintestimates_firsteq}
 &\intnt (1+s^{-\al})(1+(t-s)^{-\beta})e^{-\delta(t-s)}e^{-\gamma s} ds\nonumber
 \leq e^{-\delta t}\intnt e^{-(\gamma-\delta)s}ds +e^{-\delta t}\intnt s^{-\al} e^{-(\gamma-\delta)s}ds\nonumber\\
 &\qquad +e^{-\delta t}\intnt (t-s)^{-\beta} e^{-(\gamma-\delta)s}ds +e^{-\delta t}\intnt s^{-\al}(t-s)^{-\beta}e^{-(\gamma-\delta)s} ds.
\end{align}
In order to obtain estimates for the summands, independently of the values of $\al, \beta, \gamma, \delta$, we can start with
\begin{align*}
\intnt e^{-(\gamma-\delta)s}ds = \frac1{\gamma-\delta}[1-e^{-(\gamma-\delta)t}]\leq \frac1\eta
\end{align*}
and continue by estimating
\begin{align*}
\intnt s^{-\al} e^{-(\gamma-\delta)s}ds\leq\int_0^1s^{-\al}ds+\int_1^\infty e^{-(\gamma-\delta)s} ds \leq \frac1{1-\al}+\frac1{\gamma-\delta}\leq \frac2\eta.
\end{align*}

Also in the third term on the right hand side of \eqref{eq:lemintestimates_firsteq} we can split the integral and use the obvious estimates $(t-s)^{-\beta}\leq 1$ for $s<t-1$ and $e^{-(\gamma-\delta)(t-\sigma)}\leq e^{-(\gamma-\delta)(-\sigma)}{\leq}e^{\gamma-\delta}$ for $\sigma\in(0,1)$ to obtain
\begin{align}\nn
\intnt (t-s)^{-\beta} e^{-(\gamma-\delta)s}ds&\leq \intnt e^{-(\gamma-\delta)s} ds +\int_0^1 \sigma^{-\beta}e^{-(\gamma-\delta)(t-\sigma)}d\sigma\leq \frac{1}{\gamma-\delta}+\frac1{1-\beta}e^{\gamma-\delta}
\leq \frac{1}{\eta}+\frac1{\eta}e^{\frac1\eta}.
\end{align}
The last integral can be rewritten as
\begin{equation}\label{eq:integralrewritten}
\intnt s^{-\al}(t-s)^{-\beta}e^{-(\gamma-\delta)s} ds = t^{1-\al-\beta}\int_0^1 \sigma^{-\al}(1-\sigma)^{-\beta}e^{-(\gamma-\delta)\sigma t} d\sigma,
\end{equation}
where we have
\[
 \int_0^1 \sigma^{-\al}(1-\sigma)^{-\beta}e^{-(\gamma-\delta)\sigma t} d\sigma \leq \int_0^1 \sigma^{-\al}(1-\sigma)^{-\beta}\leq 2^\beta\int_0^{\frac12} \sigma^{-\al} d\sigma + 2^\al \int_0^{\frac12}\sigma^{-\beta}d\sigma\leq \frac2{1-\al}+\frac2{1-\beta}\leq \frac4\eta,
\]
so that \eqref{eq:integralrewritten} yields the estimate we are aiming for if $1-\alpha-\beta\leq0$ or if $t<1$ and $1-\alpha-\beta >0$.
As to $1-\al-\beta>0$ and $t\ge 1$, we estimate
\begin{align*}
 \int_0^1& \sigma^{-\al}(1-\sigma)^{-\beta}e^{-(\gamma-\delta)\sigma t} d\sigma\\
&\leq\int_0^{\frac12 t^{-\frac{1-\al-\beta}{1-\al}}}\sigma^{-\al}(1-\sigma)^{-\beta}e^{-(\gamma-\delta)\sigma t} d\sigma +\int_{\frac12 t^{-\frac{1-\al-\beta}{1-\al}}}^1 \sigma^{-\al}(1-\sigma)^{-\beta}e^{-(\gamma-\delta)\sigma t}d\sigma\\
 &\leq \left(\nicefrac12\right)^{-\beta}\int_0^{\frac12 t^{-\frac{1-\al-\beta}{1-\al}}} \sigma^{-\al} d\sigma + \left(\frac12 t^{-\frac{1-\al-\beta}{1-\al}}\right)^{-\al}e^{-(\gamma-\delta){\frac12 t^{-\frac{1-\al-\beta}{1-\al}}}}\int_{\frac12 t^{-\frac{1-\al-\beta}{1-\al}}}^1(1-\sigma)^{-\beta}d\sigma\\
 &\leq \frac{2^{\beta+\al-1}}{1-\al} t^{-(1-\al-\beta)} + \frac{2^\al}{1-\beta} t^{-(1-\al-\beta)} t^{1-\frac{\beta}{1-\al}}e^{-\frac{\gamma-\delta}2 t^{\frac\beta{1-\al}}}.
\end{align*}
Here,
\[
 t^{1-\frac{\beta}{1-\al}}e^{-\frac{\gamma-\delta}2 t^{\frac\beta{1-\al}}}\leq 1+te^{-\frac{\gamma-\delta}2 t^{\frac\beta{1-\al}}},
\]
where we have
\[
 \frac{\beta}{1-\al}\geq\beta, \qquad t^{\frac{\beta}{1-\al}}\geq t^\beta\geq t^\eta,
\]
because $t\geq 1$,
and hence
\[
 te^{-\frac{\gamma-\delta}2 t^{\frac\beta{1-\al}}}\leq te^{-\frac{\gamma-\delta}2 t^\beta} \leq te^{-\frac\eta2 t^\eta},
\]
which in combination with the finiteness of
\(
 \sup_{t>0} te^{-\frac\eta2 t^\eta}
\)
implies the assertion.
\end{proof}

In order to obtain regularity of $u$, we have employed the following result in the proof of Lemma \ref{lem:c2al}. Other than in \cite{fujiwara_morimoto}, we are concerned with the impact of the Helmholtz projection on H\"older-continuous functions (instead of on functions belonging to some $L^p$-space only.)
\begin{lem}\label{lem:projection_smoothness}
Let $\Om\subset\R^N$ be a bounded domain with $\partial\Om\in C^{1+\alpha}$ for some $\al>0$, $T>0$ let $u\in C^{\alpha,\frac\alpha2}(\Ombar\times[0,T])$ and $u=v+w$, where $\nabla\cdot v=0$ in $\Om$ and $v\cdot\nu=0$ on $\dOm$ and $w=\nabla\Phi$ for some function $\Phi$. Then $v\in C^{\alpha,\frac\alpha2}(\Ombar\times[0,T])$.
\end{lem}
\begin{proof}
 We have to find a decomposition $u=v+w$ with $\na\cdot v=0$ in $\Om$ and $v\cdot\nu=0$ on $\dOm$ and $w=\na \Phi$ for some function $\Phi$. We will construct $w$ and conclude from its smoothness that $\P u=v=u-w\in C^{\alpha,\frac\alpha2}(\Ombar\times[0,T];\R^{N})$.
 As preparation let us consider the elliptic problem
 \begin{equation}\label{eq:elliptic}
  \Delta\Phi= \na\cdot f,\qquad \delny \Phi\bdry=f\cdot\nu\bdry, \qquad \io \Phi=0.
 \end{equation}

 Only assuming $f\in C^\alpha(\Ombar)$, we fix $p>n$ and let $q$ be such that $\frac1p+\frac1q=1$. Then \cite[Thm. 4.1]{Simader_weakDirichletandNeumann}, which mirrors the usual Lax-Milgram type result in the context of $L^p$-spaces also for $p\neq 2$, asserts the existence of a unique weak solution $\Phi\in \set{\Phi\in W^{1,p}(\Om), \io \Phi=0}$ such that
 \[
  \io \na \Phi\cdot\na \Phii = \io f\na\Phii \qquad \mbox{for all } \Phii\in W^{1,q}(\Om).
 \]
 Moreover, this solution satisfies
\begin{align}\label{eq:elliptic_estimate1}
 c_1\norm[L^\infty(\Om)]{\Phi}\leq& c_2\norm[W^{1,p}(\Om)]{\Phi}\leq \norm[L^p(\Om)]{\na \Phi}\nn\\
 \leq& c_3\sup \bigset{\frac{\left\lvert \io f\na\Phii\right\rvert}{\norm[L^q(\Om)]{\na\Phii}};\, \Phii\in W^{1,q}(\Om), \na\Phii\not\equiv 0}
 \leq c_3 \norm[L^p(\Om)]{f} \leq c_4\norm[C^\al(\Ombar)]{f}
\end{align}
with positive constants $c_1$, $c_2$, $c_3$ and $c_4$ that are guaranteed to exist by the continuity of the embedding $W^{1,p}(\Om)\embeddedinto L^\infty(\Om)$, Poincar\'e's inequality, \cite[Thm. 4.1]{Simader_weakDirichletandNeumann} and continuity of the embedding $C^{\al}(\Om)\embeddedinto L^p(\Om)$, respectively.
A standard elliptic regularity result (see \cite[Thm. 2.8]{BeiHu}) moreover asserts the existence of $c_5>0$ such that $C^{1+\alpha}$-solutions $\Phi$ of \eqref{eq:elliptic} satisfy
\[
 \norm[C^{1+\al}(\Om)]{\Phi}\leq c_5(\norm[C^\al(\Ombar)]{f}+\norm[L^\infty(\Om)]{\Phi})
\]
and thus, taking into account \eqref{eq:elliptic_estimate1},
\[
 \norm[C^{1+\alpha}(\Ombar)]{\Phi} \leq c_6\norm[C^\alpha(\Ombar)]{f}
\]
with $c_6:=c_5(1 + \frac{c_4}{c_1})$.\\
Approximating $f\in C^\alpha(\Ombar)$ by a sequence of functions $\{f_n\}_{n\in\N}\sub C^\infty(\Ombar)$ for which the existence of classical solutions $\Phi_n\in C^{2+\alpha}(\Ombar)$ is asserted by well-known results (\cite[Thm. 3.3.2]{LU}), we see that for $f\in C^\alpha(\Ombar)$ problem \eqref{eq:elliptic} has a unique solution $\Phi\in C^{1+\alpha}(\Ombar)$, which moreover satisfies
\begin{equation}\label{eq:ell_c1alpha_estimate}
 \norm[C^{1+\alpha}(\Ombar)]{\Phi}\leq c_6\norm[C^\alpha(\Ombar)]{f}.
\end{equation}

For each $t$ let $\Phi(\cdot,t)$ denote the solution of
 \[
  \Delta \Phi(\cdot,t)=\na\cdot u(\cdot,t), \qquad \delny \Phi(\cdot,t)\bdry = u(\cdot,t)\cdot \nu\bdry, \qquad \io \Phi=0,
 \]
and define $w(\cdot,t):=\nabla \Phi(\cdot,t)$ and $v(\cdot,t):=u(\cdot,t)-w(\cdot,t)$, so that clearly $\na\cdot v=\na\cdot u-\na\cdot w=\na\cdot u-\Delta \Phi=0$ in $\Om$ and $v\cdot\nu=u\cdot\nu-w\cdot\nu=u\cdot\nu-\delny\Phi=0$ on $\dOm$.
Concerning smoothness, we see that $\Phi(\cdot,t)\in C^{1+\al}(\Ombar)$ entails $w(\cdot,t)\in C^{\alpha}(\Ombar)$ and for $t_1,t_2\in[0,T]$ we have that $\Phi(\cdot,t_2)-\Phi(\cdot,t_1)={:}\Psi$ solves
\[
 \Delta \Psi = \na\cdot(u(\cdot,t_2)-u(\cdot,t_1)),\qquad \delny \Psi\bdry = (u(\cdot,t_2)-u(\cdot,t_1)) \cdot \nu, \qquad \io \Psi=0
\]
so that by \eqref{eq:ell_c1alpha_estimate} 
\[
 \norm[C^{\al}(\Ombar)]{w(\cdot,t_2)-w(\cdot,t_1)}\leq \norm[C^{1+\alpha}(\Ombar)]{\Psi} \leq {c_6} \norm[C^{\al}(\Ombar)]{u(\cdot,t_2)-u(\cdot,t_1)}.
\]
By the known regularity of $u$, in conclusion we have $w\in C^{\alpha,\frac\alpha2}(\Ombar\times[0,T])$ and thus $\P u=v=u-w\in C^{\alpha,\frac\alpha2}(\Ombar\times[0,T];\R^N)$.
\end{proof}

The last statement we have postponed to this appendix is concerned with the adaptions necessary for proving Theorem \ref{thm:alternative} instead of Theorem \ref{thm:main}.

\begin{lem}\label{lem:chooseM_alternative}
Given $M, N, p_0, q_0, \beta$ as in Theorem 1.2 and some $\delta>0$, it is possible to choose $M_1$, $M_2$, $M_3$, $M_4$, $\eps>0$, $m_0<\eps|\Om|^{-\frac1{p_0}}$ such that for all $m>m_0$, for all $\alpha_1\in(\frac{m}2,\min\set{m,\lambda_1-\delta})$ and $\alpha_2\in(0,\min\set{\alpha_1,\lambda_1'-\delta})$ the inequalities
\begin{align*}
&k_7(N,q_0)+k_5(q_0)k_7(q_0,q_0)(M_1+2k_1)\norm[\Lin]{\na\Phi} C_1+3k_7({\scriptstyle \frac{N}{1+\frac{N}{q_0}}},q_0) k_5({\scriptstyle\frac{N}{1+\frac{N}{q_0}}})M_3M_4C_2\eps\leq \frac{M_3}2\\
&k_8(N,N)+k_8(N,N)k_5(N)|\Om|^{\frac{q_0-N}{Nq_0}}(M_1+2k_1)\norm[\Lin]{\na\Phi}C_3\\
&\quad\quad\quad\quad\quad\quad\quad\quad\quad\quad\quad
+3 M_3M_4k_8({\scriptstyle \frac{1}{\frac1{q_0}+\frac1N}},N) k_5({\scriptstyle \frac1{\frac1{q_0}+\frac1N}})C_4{\eps}\leq\frac{M_4}2\\
&k_3+{C_5k_2(|\Om|^{-\frac1{p_0}}+M_1+2k_1)Me^{(M_1+2k_1)\sigma\eps}\eps}+3k_2M_2M_3 C_6\eps \le \frac{M_2}2\\
&3C_SC_7k_4M_2\eps|\Om|^{-\frac{1}{p_0}}+ 3C_SC_7k_4M_2(M_1+2k_1)\eps+3(M_1+{2k_1})C_7k_4M_3\eps\le \frac{M_1}2.
\end{align*}
 hold.
\end{lem}
\begin{proof}
 The condition $m_0<\eps|\Om|^{-\frac1{p_0}}$ that is used to ensure the existence of initial data satisfying \eqref{init_small'} compells us to choose $m_0$ at the end of this proof, quite in contrast to the situation in Lemma \ref{lem:chooseM}.
 Furthermore this makes it necessary to have the estimates during the proof hold regardless of the values of $\alpha_1$, $\alpha_2$, which depend on $m$.
Fortunately, $C_1, \ldots, C_7$ indeed do not depend on $\alpha_1$, $\alpha_2$ (and thus not on $m$), but -- thanks to Lemma \ref{lem:integralestimates} -- rather on (a lower bound for) the differences between $\mu$ and $\alpha_1$, $\mu$ and $\alpha_2$ or $\lambda_1$ and $\alpha_1$. (This is the purpose $\delta$ has been introduced for.)
The only remaining parameter is $\sigma=\sigma(\alpha_1)=\int_0^\infty(1+s^{-\frac N{2p_0}}) e^{-\alpha_1 s} ds$, which is decreasing with respect to $\alpha_1$. If we decide to concentrate on relatively ``large'' values of $\alpha_1$ only, namely $\alpha_1>\frac{m}2$, (which is of no effect to the generality of Theorem 1.2), given $m>0$, for any $\alpha_1\in(\frac{m}2,\min\set{m,\lambda_1-\delta})$, we may rely on
\[
 \sigma(\alpha_1)\leq \int_0^\infty \left(1+s^{-\frac{N}{2p_0}}\right)e^{-\frac m2 s} ds \leq 2\int_0^\infty e^{-\frac m2 s}ds +\int_0^1 s^{-\frac N{2p_0}} ds \leq \frac 4m+\frac{2p_0}{2p_0-N}.
\]
We pick arbitrary $M_1>0$ and
\begin{equation}\label{eq:chooseA}
 A>(M_1+2k_1)\left(8|\Om|^{\frac1{p_0}}+\frac{1}{1-\frac{N}{2p_0}}\right).
\end{equation}
Moreover, we can choose $M_2$ such that $k_3+C_5k_2(|\Om|^{-\frac1{p_0}}+M_1+2k_1)Me^{A}A \leq \frac{M_2}4$ and $M_3$ such that $k_7(N,q_0)+k_5(q_0)k_7(q_0,q_0)(M_1+2k_1)\norm[\Lin]{\na\Phi} C_1\leq \frac{M_3}4$, and we choose $M_4$ \\ such that \mbox{$k_8(N,N)+k_8(N,N)k_5(N)|\Om|^{\frac{q_0-N}{Nq_0}}(M_1+2k_1)\norm[\Lin]{\na\Phi}C_3\leq \frac{M_4}4$.} Then we let
\begin{align*}
 0<\eps<&\min\bigg\{
 A,
\frac{1}{12k_2M_3C_6},
\frac{1}{12M_3k_8({\scriptstyle \frac{1}{\frac1{q_0}+\frac1N}},N)k_5({\scriptstyle \frac1{\frac1{q_0}+\frac1N}})C_4},\\
&\frac{1}{12k_7({\scriptstyle \frac{N}{1+\frac{N}{q_0}}},q_0)k_5({\scriptstyle\frac{N}{1+\frac{N}{q_0}}})C_2M_4},
\frac{M_1}{2( 3C_SC_7k_4M_2{(|\Om|^{-\frac{1}{p_0}}+M_1+2k_1)}+3(M_1+{2k_1})C_7k_4M_3)},
1
\bigg\}
\end{align*}
Finally, we want to choose $m_0<\eps|\Om|^{-\frac1{p_0}}$ such that $(M_1+2k_1)\sigma(\alpha_1)\eps<A$ for all $\alpha_1\in( \frac{m}2,\min\set{m,\lambda_1-\delta})$, for all $m>m_0$. This is indeed feasible, since
$\sigma(\frac\eps2|\Om|^{-\frac1{p_0}})<\frac{A}{(M_1+2k_1)\eps}$ due to
\[
 \eps\sigma\left(\frac\eps2|\Om|^{-\frac1{p_0}}\right) <\eps \left(\frac{8}{\eps|\Om|^{-\frac1{p_0}}}+\frac{2p_0}{2p_0-N}\right)\leq 8|\Om|^{\frac1{p_0}}+\frac{2p_0}{2p_0-N}<\frac{A}{M_1+2k_1}
\]
and by continuity we can find $m_0<\eps|\Om|^{-\frac1{p_0}}$ so that $\sigma(\frac{m_0}2)<\frac{A}{(M_1+2k_1)\eps}$.
With this choice, for all $\alpha_1\in( \frac{m}2,\min\set{m,\lambda_1-\delta})$, for all $m>m_0$, we have $\sigma(\alpha_1)<\sigma(\frac{m}{2})<\sigma(\frac{m_0}2)<\frac{A}{(M_1+2k_1)\eps}$.
\end{proof}

{\footnotesize

\def\cprime{$'$}

}

\end{document}